\documentclass[11pt]{article}

\usepackage[pagebackref,colorlinks=true,urlcolor=blue,linkcolor=red,citecolor=red]{hyperref}

\usepackage{epsfig, graphics, graphicx}
\usepackage{subcaption, comment, bm}
\usepackage[percent]{overpic}
\usepackage{float}
\usepackage{amssymb,bm}
\usepackage{amsthm}
\usepackage{color}
\usepackage[]{amsmath}
\usepackage[]{amsfonts}
\usepackage[]{fancyhdr}
\usepackage[]{graphicx,wrapfig}
\usepackage{enumitem}
\usepackage[utf8]{inputenc}
\usepackage[T1]{fontenc}
\usepackage{mathtools}
\usepackage{array}
\usepackage{booktabs}
\usepackage{comment}
\usepackage{amsmath}
\usepackage{tikz}
\usepackage{epigraph}
\usepackage{lipsum}
\allowdisplaybreaks

\def\th@plain{%
  \thm@notefont{}
  \itshape 
}
\def\th@definition{%
  \thm@notefont{}
  \normalfont 
}
\makeatother

\usepackage{srcltx} 
\usepackage{amscd}
\usepackage{amssymb}
\usepackage{amsmath}
\usepackage{pb-diagram}
\usepackage{graphicx}
\usepackage{hyperref}
\usepackage{array}
\usepackage[all]{xy}
\xyoption{matrix}
\xyoption{arrow}
\usepackage{pdfsync}
\usepackage{physics}
\usepackage{enumitem}
 
\usepackage{thmtools}
\usepackage{amssymb}

\DeclareUnicodeCharacter{2212}{-,+}
\usepackage[none]{hyphenat}

\theoremstyle{definition}
\newtheorem{definition}{Definition}[section]
\newtheorem{theorem}[definition]{Theorem}

\newtheorem{lemma}[definition]{Lemma}

\newtheorem{remark}[definition]{Remark}

\newcommand{\Ic}{\mathcal{I}}
\newcommand{\Jc}{\mathcal{J}}
\newcommand{\Kc}{\mathcal{K}}
\newcommand{\Lc}{\mathcal{L}}
\newcommand{\Mc}{\mathcal{M}}

\newcommand{\Rc}{\mathcal{R}}

\newcommand{\Tc}{\mathcal{T}}

\newcommand{\vx}{\textbf{\textit{x}}}

\newcommand{\vw}{\textbf{\textit{w}}}
\newcommand{\vv}{\textbf{\textit{v}}}

\newcommand{\vf}{\textbf{\textit{f}}}

\newcommand{\vg}{\textbf{\textit{g}}}

\newcommand{\D}{\mathrm{d}}
\newcommand{\DR}{\mathbb{D}_R}
\newcommand{\vuu}{\textbf{\textit{u}}}

\title{\vspace{-1cm} V-Line Tensor Tomography in a Disk: Theoretical and Numerical Reconstruction}

\author{Rahul Bhardwaj\thanks{Department of Mathematics, Indian Institute of Technology, Ropar, Punjab - 140001, India. \url{bhardwaj161067@gmail.com}} \and Madhu Gupta\thanks{Department of Mathematics, Indian Institute of Technology, Gandhinagar, Gujarat - 382355, India. \url{madhu.gupta@iitgn.ac.in}}}
\date{}
\begin{document}
\maketitle
\begin{abstract}
In this article, we investigate V-line transforms for symmetric $m$-tensor fields whose support lies inside a disk of radius $R$ and centered at the origin. We provide an explicit characterization of the kernel of the V-line transforms acting on a symmetric $m$-tensor field and derive a new inversion formula using a decomposition result. In addition, we present a comprehensive numerical verification and validation of the inversion algorithms for these V-line transforms for vector fields and symmetric $2$-tensor fields, which were recently developed in \cite{bhardwaj_2024,bhardwaj2025tensor}. The reconstruction results obtained for various phantoms demonstrate the effectiveness and robustness of the proposed numerical methods, including in the presence of noise. 
\end{abstract}

\vspace{0.5mm}
\section{Introduction}
Optical tomography has emerged as an important and rapidly developing imaging modality in modern imaging science. It uses transmitted and scattered light to obtain information about an object's internal structure and physical properties. In most imaging applications, the objective is to reconstruct spatially varying absorption and scattering coefficients from measurements collected at the object's boundary. Such reconstruction problems naturally lead to inverse problems involving integral geometric transforms. Under suitable physical assumptions, it is commonly assumed that most photons undergo only a single scattering event while propagating through the object (see \cite{Florescu2009,Florescu2010,Florescu2011}). This single-scattering behavior naturally leads to the terminology \emph{broken-ray} or \emph{V-line transform}. 

Over the past several years, significant attention has been devoted to imaging models based on generalized Radon transforms, particularly the V-line transform.
These transforms generalize the classical Radon transform by integrating functions or tensor fields over broken-ray trajectories composed of two straight segments that intersect at a common vertex. Depending on the underlying physical model, the unknown quantity may be represented by a scalar function, vector field, symmetric $2$-tensor field, or, more generally, a higher-order symmetric tensor field. Consequently, the inversion and analysis of these V-line transforms are fundamental for the corresponding image reconstruction problems. V-line transforms naturally arise in several imaging modalities, including single-scattering optical tomography \cite{Florescu2010,Florescu2009}, single-scattering X-ray tomography \cite{Kats_Krylov-15, walker2021iterative}, fluorescence imaging \cite{florescu2018}, Compton camera imaging \cite{r8,Basko1997FullyTD}, Compton-scattering emission tomography with collimated detectors \cite{r9,r10}, and gamma-ray transmission/emission imaging \cite{Rigaud2013CombinedMO}.

Due to their broad range of applications, V-line transforms have been extensively investigated from both theoretical and computational points of view; see, for example, \cite{Ambartsoumian2019,amb-lat_2019,Amb_Lat_star,ambartsoumian2007thermoacoustic,Gouia_Amb_V-line,Ilmavirta,Ilmavirta2018BrokenRT,Ilmavirta_Mikko,Shubham,Kats_Krylov-13,Sherson,walker2019broken,Palamodov2017} and the references therein. Previous studies on V-line transforms have primarily focused on two distinct geometric configurations. The first class consists of V-lines whose vertices lie on or outside the boundary of the imaging domain; see, for instance, \cite{Terz-review-18, Indrani_2024}. The second class consists of V-line trajectories whose vertices are located within the support of the unknown function. Such transforms naturally arise in single-scattering tomography and Compton camera imaging. The present work concerns this second class of V-line transforms, and we briefly discuss relevant results in the next paragraph. 

In \cite{Gaik_Mohammad_Rohit,Gaik_Rohit_Indrani}, the authors investigated V-line transforms for vector fields and symmetric $2$-tensor fields in a setting where the vertex moves inside the disk while the V-line directions remain fixed. They derived exact kernel characterizations for these transforms and showed that vector fields and symmetric $2$-tensor fields can be uniquely reconstructed from appropriate combinations of V-line transform measurements and their corresponding moments. The numerical implementation, verification, and validation of these reconstruction algorithms were later developed in \cite{Gaik_Latifi_Rohit,Ambartsoumian2024VlineTT}.

A different geometry of V-lines whose vertex located inside the disk, described in Section~\ref{sec:Definition and notations}, was introduced by \cite{Ambarsoumian_2012}. In that work, the author derived theoretical and numerical reconstruction formulas for the V-line Radon transform of scalar functions under the support restriction that lies inside a disk. Subsequently, in \cite{Ambartsoumian_2013}, a complete reconstruction procedure was established based on a series-type inversion formula for the same transform. The numerical verification of these inversion methods was further studied in \cite{ambartsoumian2016numerical}.
For a detailed account of the mathematical foundations, historical development, and applications of the generalized Radon transforms, we refer the reader to the recent monograph~\cite{amb-book} by Gaik Ambartsoumian. The works \cite{bhardwaj_2024,bhardwaj2025tensor} extended the reconstruction results of \cite{Ambarsoumian_2012,Ambartsoumian_2013} from scalar functions to vector fields and higher-order symmetric tensor fields respectively. In these studies, inversion formulas were derived for the recovery of vector fields and symmetric tensor fields from appropriate collections of V-line transform data.

Motivated by all these developments, the present article makes several contributions. First, we establish an explicit characterization of the kernels of the mixed V-line transforms acting on symmetric $m$-tensor fields. Based on a suitable decomposition of tensor fields, we derive new inversion formulas for reconstructing special classes of symmetric $m$-tensor fields from the corresponding V-line transform data. These theoretical results extend the current understanding of V-line tensor tomography and provide new insights into the structure of these integral transforms. 

In addition, we provide a detailed numerical verification and validation of the inversion algorithms corresponding to the theoretical reconstruction formulas established in \cite{bhardwaj_2024,bhardwaj2025tensor} for vector field and symmetric $2$-tensor fields. The performance of the reconstruction procedures is examined using a variety of representative phantoms, and all numerical experiments are carried out in MATLAB. Furthermore, reconstructions are performed under different levels of additive noise to assess the practical performance of the proposed algorithms.

The remainder of this article is organized as follows. In Section~\ref{sec:Definition and notations}, we introduce the straight line and V-line transforms considered in this work together with the notation used throughout the article. Section~\ref{sec:main results} presents the main theoretical results of this article. Section~\ref{sec; kernel description} is devoted to a detailed proof of the main results. Section~\ref{sec: Numerical implementation} describes the phantoms used in the numerical experiments and explains the procedure for generating the forward data. This section is further divided into two subsections, Subsections~\ref{Sec: Reconstruction of vector field} and Subsections~\ref{Subsec: Rec Tensor}, which focus on the reconstruction of vector fields and symmetric $2$-tensor fields, respectively. Finally, we conclude the article with acknowledgements in Section~\ref {sec:acknowledgements}.
 \section{Definition, Notations and Preliminaries}\label{sec:Definition and notations}
This introductory section presents the notation and definitions used throughout this article. The regular fonts are used to denote scalars or scalar-valued functions (for example, $x_1$, $x_2$, $f$, and $g$, etc), whereas bold fonts are used to represent vectors, vector fields, or tensor fields in $\mathbb{R}^2$ (for example, $\vf$, $\vx$, $\vv$, etc).
\vspace{2mm}\\
\noindent The disk of radius $R$ centered at the origin is denoted by $\mathbb{D}_R$, and its boundary is represented by $\partial \mathbb{D}_R$. We denote $C_0^\infty(\DR, S^m(\mathbb{R}^2))$ to be the space of the symmetric $m$-tensor fields whose components belong to $C_0^\infty(\DR)$, the space of smooth functions compactly supported in $\DR$. Let $\omega^{1}, \omega^{2},\dots, \omega^{m}$ are vectors in $\mathbb{R}^2$, then the tensor product $\omega^{1}\otimes \omega^{2}\otimes \dots\otimes  \omega^{m}$ is defined component-wise by
 \begin{align}
    \left( \omega^{1}\otimes \omega^{2}\otimes \dots\otimes  \omega^{m}\right)_{i_1i_2\cdots i_m} = \omega^{1}_{i_1}\omega^{2}_{i_2}\dots \omega^{m}_{i_m}, \qquad i_1,i_2,\cdots,i_m \in \{1, 2\}.
 \end{align}
 The corresponding symmetric tensor product is defined by 
 \begin{align}
     \omega^{1}\omega^{2}\dots \omega^{m} = \frac{1}{m!}\sum\limits_{\sigma\in \Pi_{m}} \omega^{\sigma(1)}\otimes \omega^{\sigma(2)}\otimes \dots\otimes  \omega^{\sigma(m)},
 \end{align}
 where $\Pi_m$ denotes the symmetric group consisting of all permutations of the set $\{1,2,\ldots,m\}$, for more deetails see \cite[Section 2.1 ]{Sharafutdinov_Book}. The standard inner product on $S^{m}(\mathbb{R}^{2})$ is defined by  \begin{align*}
    \left\langle \vf,\vg \right\rangle = f_{i_1i_2\cdots i_m}g^{i_1i_2\cdots i_m} 
\end{align*}
and the corresponding norm is denoted by $|\cdot|$. Throughout this article, we use the \textquotedblleft{\textbf{Einstein summation convention}}'' according to which summation
over repeated indices is understood automatically from $1$ to $2$. F
or $\vuu=(u^1,u^2)\in\mathbb{R}^2$, the notation $\vuu^m$ represents the symmetric tensor obtained by taking the $m$-fold symmetric product of $\vuu$ with itself. Its components are given by
\begin{align*}
    (\vuu^m)_{i_1\cdots i_m}
=
u^{i_1}\cdots u^{i_m},
\qquad
i_1,\ldots,i_m\in\{1,2\}.
\end{align*}
Thus, we have $\vuu^m \in S^{m}(\mathbb{R}^2)$.
\vspace{2mm}\\
\noindent Next, we introduce the differential operators used later in this article. Compositions of the classical differential operators, namely the gradient and orthogonal gradient operators, naturally give rise to the inner differentiation and inner orthogonal differentiation operators acting on symmetric $m$-tensor fields $\vf$. We also define the divergence and orthogonal divergence operators for symmetric $m$-tensor fields. These operators are given as follows:
\begin{itemize}
\item The inner differentiation and inner orthogonal differentiation operators $\mathrm{d},\D^\perp :C_0^\infty(\DR,S^m(\mathbb{R}^2)) \rightarrow  C_0^\infty\left(\mathbb{D}_R, S^{m+1}(\mathbb{R}^2)\right)$ are defined as follows:
\begin{align*}
    ( \D \vf )_{i_1\dots i_mj} &:= \frac{1}{m+1}\left(\frac{\partial{f_{i_1\dots i_m}}}{\partial x_j}+\sum\limits_{k=1}^m\frac{\partial{f_{i_1\dots i_{k-1}ji_{k+1}\dots i_m}}}{\partial x_{i_k}}\right) \quad \mbox{ and } \\ (\D^\perp\vf)_{i_1\dots i_mj} &:= \frac{1}{m+1}\left((-1)^j\frac{\partial{f_{i_1\dots i_m}}}{\partial x_{3-j}}+\sum\limits_{k=1}^m(-1)^{i_k}\frac{\partial{f_{i_1\dots i_{k-1}ji_{k+1}\dots i_m}}}{\partial x_{3-i_k}}\right).
\end{align*}
\item The divergence and the orthogonal divergence operators $\delta, \delta^\perp:C_0^\infty\left(\DR,S^{m}(\mathbb{R}^2)\right) \rightarrow  C_0^\infty\left(\DR,S^{m-1}(\mathbb{R}^2)\right)$ are defined as follows:
\begin{align*}
    (\delta \vf)_{i_1\dots i_{m-1}} &:= \frac{\partial{f_{i_1\dots i_{m-1}1}}}{\partial x_1}+\frac{\partial{f_{i_1\dots i_{m-1}2}}}{\partial x_2} \quad \mbox{ and }  \\(\delta^\perp \vf)_{i_1\dots i_{m-1}} &:= \frac{\partial{f_{i_1\dots i_{m-1}2}}}{\partial x_1} - \frac{\partial{f_{i_1\dots i_{m-1}1}}}{\partial x_2}.
\end{align*}
\end{itemize}

\noindent We now recall the definitions of the V-line (or broken ray) and the corresponding V-line transforms for symmetric $m$-tensor fields, which were previously studied in \cite{Ambarsoumian_2012,bhardwaj_2024,bhardwaj2025tensor}. For the definitions and background on straight-line transforms, we refer the reader to \cite{Sharafutdinov_Book,derevtsov2015tomography}.
\begin{definition}[Broken-ray]
   Let $\DR$ be the disk of radius $R$ centered at the origin, and let $\theta \in (0, \pi/2)$ be a fixed scattering angle. The \textbf{V-line} or \textbf{broken ray}, denoted by $BR(\beta, d)$, consists of two connected straight-line segments which starts at the boundary point $\vx_\beta$ lies on $\partial \DR$, and travels a distance $d$ in the radial direction $\vuu_\beta$, and then breaks into another ray through the angle $\pi - \theta$, after which it continues along the
direction $\vv_\beta$ (see Figure \ref{fig: def of broken ray}). 
\end{definition}
\noindent More precisely, $BR(\beta, d)$ is defined by:
\begin{align}\label{eq:definition of BR(beta,d)}
    BR(\beta,d) =  \left\{\vx_\beta + t \vuu_\beta: 0\leq t \leq d\right\}  \cup \left\{\vx_\beta + d \vuu_\beta + s \vv_\beta: 0 \leq s < \infty\right\}.
\end{align} 

\begin{figure}[H]
\centering
\begin{subfigure}[b]{0.45\textwidth}
\centering
\includegraphics[width=\textwidth]{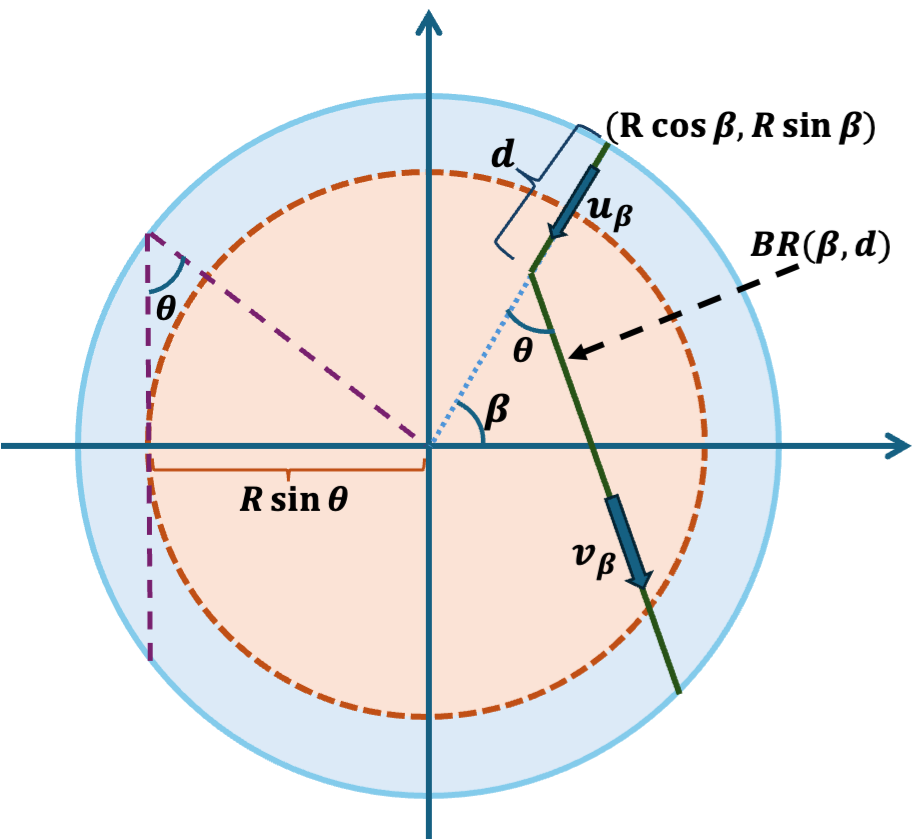}\caption{}\label{fig: def of broken ray}
\end{subfigure}
\hfill
\begin{subfigure}[b]{0.45\textwidth}
\centering
\includegraphics[width=\textwidth]{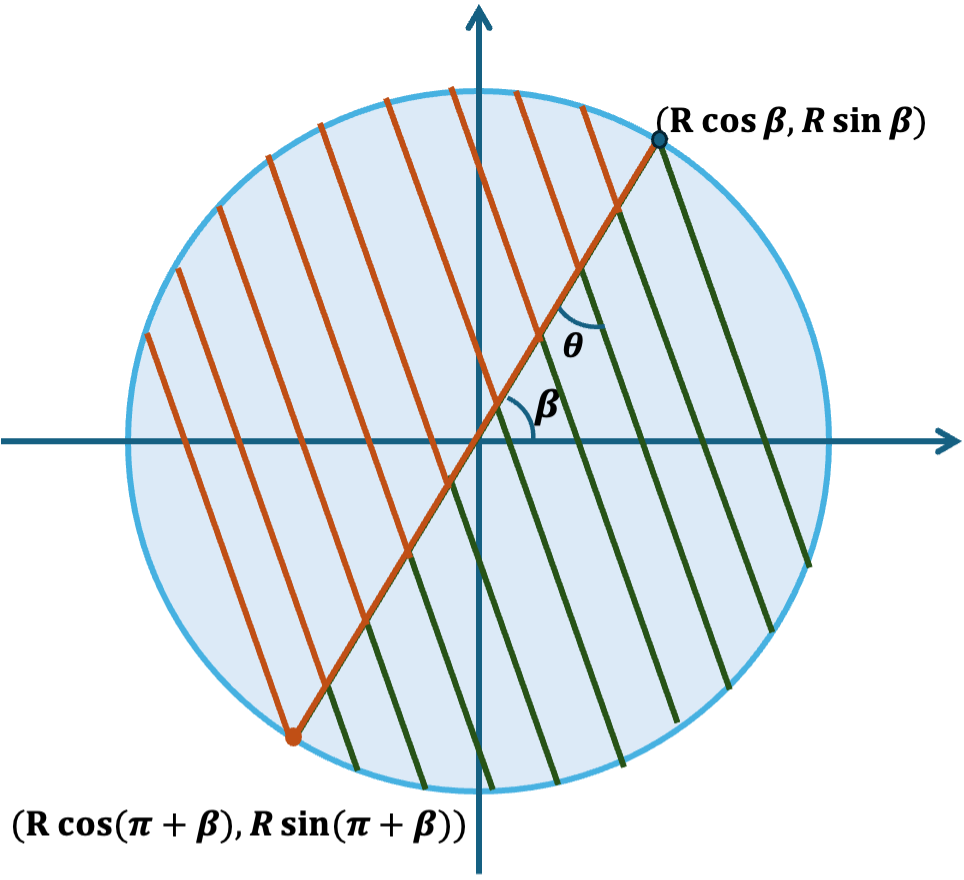}
\caption{}\label{fig: data is taken over}
\end{subfigure}
\caption{(a) Construction of the V-line $BR(\beta,d)$ \& (b) Configuration of broken rays used for data collection.}\label{fig: Description of V-lines}
\end{figure}
\noindent Here $\vx_\beta = (R \cos \beta, R \sin \beta)$ denotes the point on the boundary $\partial\DR$ from which the broken ray originates. The vectors
  $\vuu_\beta =  -(\cos \beta, \sin \beta)$ and $\vv_\beta = -\left(\cos (\theta + \beta), \sin (\theta +\beta)\right)$ are unit vectors representing the directions of the first and second segments of the V-line, respectively. Their corresponding orthogonal unit vectors are given by
  $\vuu_{\beta}^\perp = (\sin \beta, - \cos \beta)$ and $\vv_\beta^{\perp} = \left(\sin (\theta +\beta), -\cos (\theta + \beta)\right)$ respectively. For each pair $(\beta,d)$, the unit vectors $\vuu_\beta$ and $\vv_\beta$ are uniquely determined and will be used throughout the subsequent definitions of the V-line transforms acting on symmetric $m$-tensor fields.
\begin{definition}\label{def:mixed v-line transform}
For $\vf \in C_0^{\infty}(\DR; S^{m}(\mathbb{R}^{2}))$ and $0 \leq k \leq m$, the \textbf{kth mixed V-line transform} of $\vf$ is given as follows
\begin{equation}\label{eq:def en as followsMVT}
\mathcal{M}^{(k)}\vf (\beta,d) := \int_{0}^{d} \left\langle\vf (\vx_\beta +s\vuu_\beta),(\vuu_\beta^\perp)^{k}\vuu_\beta^{m-k} \right\rangle\,ds + \int_{0}^{\infty}  \left\langle \vf(\vx_\beta + d\vuu_\beta + s\vv_\beta), (\vv_\beta^\perp)^{k}\vv_\beta^{m-k}\right\rangle\,ds, 
\end{equation}
where $(\vuu_\beta^\perp)^{k}\vuu_\beta^{m-k} := (\vuu_\beta^{\perp})^{i_1}(\vuu_\beta^{\perp})^{i_2}\cdots (\vuu_\beta^{\perp})^{i_k}(\vuu_\beta^{\perp})^{i_{k+1}}\cdots \vuu_\beta^{i_m} $, for $i_1,i_2,\cdots ,i_m = 1, 2$, $\beta \in [0, 2\pi)$ and $ d \in [0,2R]$. 
\end{definition}
\begin{remark}
Two extreme values of $k$ correspond to special transforms that we to mention here; we will use them in the numerical validation sections. 
\begin{itemize}
    \item For $k=0$, the above transform $\mathcal{M}^{(0)}\vf $ is known as the longitudinal V-line transform, and it is denoted by $\Lc\vf$.
    \item For $k=m$, the above transform $\mathcal{M}^{(m)}\vf $ is known as the transverse V-line transform, and it is denoted by $\Tc\vf$.
\end{itemize}
\end{remark}
 \noindent  Subsequently, we also introduce the straight-line ray transforms for the symmetric $m$-tensor field, which will be used multiple times later in this article.
These transforms are defined by integrating tensor fields along straight lines in $\mathbb{R}^{2}$. For $\psi\in[0,2\pi)$ and $p\in\mathbb{R}$, let
\begin{align*}
    L(\psi,p):=\{(y_1,y_2)\in\mathbb{R}^{2}:y_1\cos\psi+y_2\sin\psi=p\}
\end{align*}
denote the straight line whose signed distance from the origin is $p$ and whose normal vector is
$\vw=(\cos\psi,\sin\psi)$.
The corresponding unit tangent vector is
$\vw^\perp=(-\sin\psi,\cos\psi),$
which is orthogonal to \(\vw\). The line $L(\psi,p)$ can therefore be parameterized as
$\{\,p\vw+s\vw^\perp:\; s\in\mathbb{R}\,\}$.
\begin{definition}\label{def:mixed ray transform}
For $\vf \in C_0^{\infty}(\DR; S^{m}(\mathbb{R}^{2}))$ and $0 \leq k \leq m$, the \textbf{kth mixed ray transform} of $\vf$ is given as follows
\begin{equation}\label{eq:;def MRT}
\mathcal{K}^{(k)}\vf (\psi,p) = \mathcal{K}^{(k)}\vf (\vw,p) := \int_\mathbb{R} \left\langle \vf (p \vw + s\vw^\perp), \vw^{k}(\vw^\perp)^{m-k}\right\rangle\,ds, 
\end{equation}
where $\vw^{k}(\vw^\perp)^{m-k} := w^{i_1}w^{i_2}\cdots w^{i_k}(w^{\perp})^{i_{k+1}}\cdots (w^{\perp})^{i_{m}} $, for $i_1,i_2,\cdots ,i_m = 1, 2$, $\psi \in [0, 2\pi)$ and $p \in \mathbb{R}$.
\end{definition}
\begin{remark}
As in the $V$-line setup, the special cases $k=0$ and $k=m$ correspond to the longitudinal (denoted by $\Ic\vf$) and transverse (denoted by $\Jc\vf$)  ray transforms, respectively. 
\end{remark}
\noindent Finally, we state the Radon transform and its inversion.
\begin{definition}\label{def:the Radon transform}
Let $f$ be a scalar function field in  $C_0^\infty(\DR)$. The \textbf{Radon transform} of $f$ is given as follows
\begin{equation}\label{eq:def Radon transform}
\Rc f (\psi,p) = \Rc f (\vw,p) := \int_\mathbb{R} f (p \vw + s\vw^\perp)\,ds, \quad \psi \in [0, 2\pi) \mbox { and } p \in \mathbb{R}.
\end{equation}
\end{definition}
\noindent The injectivity of the Radon transform implies that $f$ can be reconstructed uniquely from its Radon data. An explicit inversion formula is given by (see \cite{Helgason_Book}):
\begin{align}\label{eq:Radon's inversion formula}
    f(\vx) = \frac{1}{2\pi} \left(-\Delta\right)^{1/2} \int_0^{2\pi}\Rc f((\cos \alpha, \sin \alpha), x \cos \alpha + y \sin \alpha) d\alpha. 
\end{align}
Next, we present the decomposition results for symmetric $m$-tensor fields that will be needed later.
\begin{theorem}[Tensor Field Decomposition {\cite[Theorem~5.1]{derevtsov2015tomography}}]
Let $\vf \in C_0^{\infty}(\DR; S^{m}(\mathbb{R}^{2}))$ be a symmetric $m$-tensor field. Then there exist scalar potentials $\Psi^{(0)}, \Psi^{(1)}, \dots, \Psi^{(m)} $ $ \in C^{\infty}(\DR)$ satisfy the following  boundary conditions $\D^{\,k}\Psi^{(j)}\big|_{\partial\DR}=0,\
k=0,\dots,m-1$, such that 
\begin{equation}\label{eq:decomposition}
\vf =  (\D^{\perp})^m \Psi^{(0)} + \sum_{j=1}^m (\D^{\perp})^{m-j} \D^j \Psi^{(j)}.
\end{equation}
\end{theorem}
\begin{lemma}[{\cite[Lemma~5.1]{derevtsov2015tomography}}]\label{Lemma: commute}
Let $\vf \in C_0^{\infty}(\DR; S^{m}(\mathbb{R}^{2}))$ be a symmetric $m$-tensor field. Then the operators $\D$ and $\D^\perp$ commute, that is,
\begin{align*}
    \D^\perp(\D\vf)=\D(\D^\perp\vf).
\end{align*}
\end{lemma} 
\section{Theoretical Results}\label{sec:main results}
This section is devoted to the statement of the main results of the article. We start by presenting the kernel characterisation of the mixed V-line transforms for symmetric $m$-tensor fields. We show that each of these transforms possesses a non-trivial kernel. This behavior
is also analogous to the case of the straight-line integral transforms of symmetric $m$-tensor fields, where the corresponding kernel characterizations are
known to be exact. The next theorem describes the kernel structure of the V-line transforms associated with symmetric $m$-tensor fields.
\begin{theorem}\label{Thm: kernel}
Let $\vf \in C_{0}^{\infty}(\mathbb{D}_{R\sin{\theta}},S^m(\mathbb{R}^2))$. Then the symmetric $m$-tensor field $\vf$ belongs to the kernel of the kth-mixed V-line transform
 $ \Mc^{(k)}\vf(\beta,d),\ 0\leq k\leq m,$
if and only if there exist scalar potentials
$\Psi^{(j)} \in C^{\infty}(\DR),$ $0\leq j\leq m$ $(j\neq k)$,
satisfying the boundary conditions
$\D^{\,\ell}\Psi^{(j)}\big|_{\partial\DR}=0,
\
\ell=0,\dots,m-1,$
such that
\begin{align}\label{eq;decomposition}
\vf = \sum_{\substack{j=0 \\ j\neq k}}^{m}
(\D^{\perp})^{\,m-j}\D^{\,j}\Psi^{(j)}.
\end{align}
\end{theorem}

\noindent Next, we recall a reconstruction result for symmetric $m$-tensor fields from V-line transform data, previously obtained in \cite{bhardwaj_2024,bhardwaj2025tensor}. The numerical implementation, verification, and validation of the associated inversion algorithms for vector fields and symmetric $2$-tensor fields are presented in Section~\ref{sec: Numerical implementation}.

\begin{theorem}{\cite[Theorem~2.9]{bhardwaj2025tensor}}\label{th:full data recovery}
Let $\vf \in C_0^{\infty}(\mathbb{D}_{R\sin{\theta}}; S^{m}(\mathbb{R}^{2}))$. Then $\vf$ is uniquely recovered from the information of its mixed V-line transforms data, $\Mc^{(k)} \vf (\beta, d)$ ($0\leq k \leq m$), where $\beta \in [0, 2\pi)$ and $d \in [0, 2R]$. 
\end{theorem}

\noindent The next result provides the inversion formulas for the V-line transforms associated with special classes of symmetric $m$-tensor fields. More precisely, we consider tensor fields of the form
\begin{align*}
    \vf=(\D^\perp)^{m-k}\D^k\Psi^{(k)},
\qquad 0\leq k\leq m,
\end{align*}
where $\Psi^{(k)}$ is a scalar potential. For this particular class of symmetric $m$-tensor fields, complete reconstruction can be achieved using only the $k$-th mixed V-line transform data. This is due to the fact that all remaining components of the tensor field belong to the kernel of the corresponding $k$-th mixed V-line transform, which was established in Theorem~\ref{Thm: kernel}. Consequently, for this special class of symmetric tensor fields, a single mixed V-line transform contains sufficient information for complete recovery.

\begin{theorem}\label{thm: special type tensor fields}
    Let $0\leq k\leq m$, and let $\Psi^{(k)}\in C_{0}^{\infty}(\DR)$
satisfy the boundary conditions
$\D^{\,\ell}\Psi^{(k)}\big|_{\partial\DR}=0,
\
\ell=0,\dots,m-1$. If $\vf \in C_0^\infty(\mathbb{D}_{R\sin{\theta}},S^m(\mathbb{R}^2))$ is a symmetric $m$-tensor field of the form $\vf = 
(\D^{\perp})^{\,m-k}\D^{\,k}\Psi^{(k)}$. Then, the tensor field $\vf$ can be reconstructed explicitly from the knowledge of the $k$-th mixed V-line transform $\Mc^{(k)}\vf$.
\end{theorem}

\noindent To prove this result, our approach converts the information obtained from the V-line transform data into corresponding straight-line transform data. Once this connection is established, the inversion formulas for straight-line transforms can be applied to reconstruct the unknown vector field or symmetric $m$-tensor field.

\noindent It is important to note that all these theorems require certain support restrictions on $\vf$, which depend on the fixed scattering angle $\theta$. These restrictions arise naturally from the reconstruction method employed in the proofs. In particular, as illustrated in Figure~\ref{fig: data is taken over}, straight-line transform data corresponding to lines outside the disk of radius $R\sin{\theta}$ cannot be generated from the available V-line data. This reconstruction strategy was originally introduced by Gaik in \cite{Ambarsoumian_2012} for scalar functions. 

\noindent The detailed proofs of  Theorem~\ref{th:full data recovery} for the vector field and the symmetric $m$-tensor field can be found in \cite{bhardwaj_2024,bhardwaj2025tensor}. The results stated in Theorem~\ref{Thm: kernel} and Theorem~\ref{thm: special type tensor fields} are new. Detailed proofs of both theorems are provided in the Section~\ref{sec; kernel description}.

\section{Proofs of the Main Results}\label{sec; kernel description}
This section is devoted to the proofs of Theorems~3.1 and~3.3. These results establish the kernel characterization of the mixed V-line transforms acting on symmetric $m$-tensor fields and derive inversion formulas for the reconstruction of special classes of symmetric $m$-tensor fields from the corresponding V-line transform data.
\subsection{Proof of Theorem~\ref{Thm: kernel}}
   \begin{proof}
        From the proof of Theorem \ref{th:full data recovery} (see \cite[Section~3]{bhardwaj2025tensor}), for each $0\leq k\leq m,$ the mixed V-line transforms $ \Mc^{(k)}\vf(\beta,d)$ is related to the corresponding straight-line transforms $\mathcal{K}^{(k)}\vf $ through the following identity:
    \begin{align}\label{relation btw V-line and st line transform for tensor field}
        \Kc^{(k)}\vf(\psi_{\beta},t_{d}) = \Mc^{(k)}\vf(\beta,d) + (-1)^{m-k}\Mc^{(k)}\vf(\beta + \pi,2R-d) - \Mc^{(k)}\vf({\beta,2R}),
    \end{align}
    where  $d\in[0,2R]$, $\beta \in [0, 2\pi)$, $\psi_{\beta} = \beta + \theta + \pi/2$ denotes the polar angle of the corresponding straight line, and  $t_{d} = (R-d)\sin( \theta)$ represents its signed distance from the origin.
    
    \noindent Now suppose that, for a fixed $k\in\{0,\dots,m\}$, the mixed ray transforms vanishes, i.e.,
    \begin{align}
     \Mc^{(k)}\vf(\beta,d)=0, \quad \text{for all  $\beta \in [0, 2\pi)$ and $d\in[0,2R]$}.
\end{align} Then, using the relations in \eqref{relation btw V-line and st line transform for tensor field}, we immediately obtain
\begin{align}
\Kc^{(k)}\vf(\psi_\beta,t_d)=0,
    \quad \text{for all $(\psi_\beta,t_d)$.}
\end{align} 
Thus, we have
\begin{align}
        \Kc^{(k)}\vf=0, \quad 0\leq k\leq m.
\end{align}
We now use the classical kernel characterization of the mixed straight-line transforms for symmetric $m$-tensor fields (see \cite{Sharafutdinov_Book}, or  \cite[Section~5]{derevtsov2015tomography}). Specifically, these characterizations are given by
\begin{align}\label{kernel of IF, KF and JF}
\Kc^{(k)}\vf = 0
\quad \Longleftrightarrow \quad
\exists \ \Psi^{(j)} \in C^\infty(\DR),
\quad j=0,\dots,m, \ j\neq k,
\end{align}
such that
\begin{align}
\vf = \sum_{\substack{j=0 \\ j\neq k}}^{m}
(\D^{\perp})^{\,m-j}\D^{\,j}\Psi^{(j)},
\qquad
\D^{\,\ell}\Psi^{(j)}\big|_{\partial\DR}=0,
\ \
\ell=0,\dots,m-1.
\end{align}
Therefore, combining the above characterization with the relation between the V-line and straight-line transforms, we conclude that 
\begin{align}\label{kernel of Lf, Mf, Tf}
\Mc^{(k)}\vf = 0
\quad &\Longrightarrow \quad
\exists \ \Psi^{(j)} \in C^\infty(\DR),
\quad j=0,\dots,m, \ j\neq k,
\end{align}
such that
\begin{align}
\vf = \sum_{\substack{j=0 \\ j\neq k}}^{m}
(\D^{\perp})^{\,m-j}\D^{\,j}\Psi^{(j)},
\qquad
\D^{\,\ell}\Psi^{(j)}\big|_{\partial\DR}=0,
\ \
\ell=0,\dots,m-1.
\end{align}
\noindent Conversely, suppose that for a symmetric $m$-tensor field $\vf \in C_0^\infty(\DR,S^m(\mathbb{R}^2))$, there exist scalar functions $\Psi^{(j)} \in C^\infty(\DR),$ $j=0,\dots,m, \ j\neq k,$ such that
\begin{align}\label{eq:decomposition converse}
\vf = \sum_{\substack{j=0 \\ j\neq k}}^{m}
(\D^{\perp})^{\,m-j}\D^{\,j}\Psi^{(j)},
\qquad
\D^{\,\ell}\Psi^{(j)}\big|_{\partial\DR}=0,
\ \
\ell=0,\dots,m-1.
\end{align}
By the definition of the mixed V-line transform acting on symmetric $m$-tensor field, $\text{for all}\ 
 \ \beta \in [0, 2\pi) \mbox { and } d \in [0,2R]$, we have 
 \begin{align*}
     \Mc^{(k)}\vf (\beta,d) &= \int_{0}^{d} \left\langle\vf (\vx_\beta +s\vuu_\beta),(\vuu_\beta^\perp)^{k}\vuu_\beta^{m-k} \right\rangle\,ds + \int_{0}^{\infty}  \left\langle \vf(\vx_\beta + d\vuu_\beta + s\vv_\beta), (\vv_\beta^\perp)^{k}\vv_\beta^{m-k}\right\rangle\,ds.
     \end{align*}
     Substituting the decomposition \eqref{eq:decomposition converse} into the above definitions gives
     \begin{align*}
    \Mc^{(k)}\vf (\beta,d) &= \int_{0}^{d} \left\langle \left(\sum_{\substack{k\neq j=0}}^{m}
(\D^{\perp})^{\,m-j}\D^{\,j}\Psi^{(j)}\right) (\vx_\beta +s\vuu_\beta), (\vuu_\beta^\perp)^{k}\vuu_\beta^{m-k} \right\rangle \,ds\\[2pt]
& \qquad + \int_{0}^{\infty} \left\langle\left(\sum_{\substack{k\neq j=0}}^{m}
(\D^{\perp})^{\,m-j}\D^{\,j}\Psi^{(j)}\right)(\vx_\beta + d\vuu_\beta + s\vv_\beta), (\vv_\beta^\perp)^{k}\vv_\beta^{m-k} \right\rangle\,ds\\
     &= \int_{0}^{d} \left\langle \left(
(\D^{\perp})^{\,m}\Psi^{(0)}\right) (\vx_\beta +s\vuu_\beta), (\vuu_\beta^\perp)^{k}\vuu_\beta^{m-k} \right\rangle \,ds\\[2pt]
&\qquad+ \sum_{\substack{j=1 \\ j\neq k}}^{m}\int_{0}^{d} \left\langle \left(
(\D^{\perp})^{\,m-j}\D^{\,j}\Psi^{(j)}\right) (\vx_\beta +s\vuu_\beta), (\vuu_\beta^\perp)^{k}\vuu_\beta^{m-k} \right\rangle \,ds\\[2pt]
& \qquad + \int_{0}^{\infty} \left\langle\left(
(\D^{\perp})^{\,m}\Psi^{(0)}\right)(\vx_\beta + d\vuu_\beta + s\vv_\beta), (\vv_\beta^\perp)^{k}\vv_\beta^{m-k} \right\rangle\,ds\\[2pt]
& \qquad + \sum_{\substack{j=1 \\ j\neq k}}^{m}\int_{0}^{\infty} \left\langle\left(
(\D^{\perp})^{\,m-j}\D^{\,j}\Psi^{(j)}\right)(\vx_\beta + d\vuu_\beta + s\vv_\beta), (\vv_\beta^\perp)^{k}\vv_\beta^{m-k} \right\rangle\,ds.
 \end{align*}
Next, we consider three cases.

\noindent\textbf{Case 1: $1\leq k\leq m-1$.} 
 Consequently, by using Lemma~\ref{Lemma: commute} together with an application of the chain rule, each integrand can be expressed as a total derivative with respect to the line parameter $s$. More precisely, we have
 \begin{align*}
 \Mc^{(k)}\vf (\beta,d)&= \int_{0}^{d} \dfrac{d}{ds}\left\langle \left(
(\D^{\perp})^{\,m-1}\Psi^{(0)}\right) (\vx_\beta +s\vuu_\beta), (\vuu_\beta^\perp)^{k-1}\vuu_\beta^{m-k} \right\rangle \,ds\\[2pt]
&\qquad + \sum_{\substack{j=1 \\ j\neq k}}^{m}\int_{0}^{d} \dfrac{d}{ds}\left\langle \left(
(\D^{\perp})^{\,m-j}\D^{\,j-1}\Psi^{(j)}\right) (\vx_\beta +s\vuu_\beta), (\vuu_\beta^\perp)^{k}\vuu_\beta^{m-k-1} \right\rangle \,ds\\[2pt]
& \qquad + \int_{0}^{\infty} \dfrac{d}{ds} \left\langle\left(
(\D^{\perp})^{\,m-1}\Psi^{(0)}\right)(\vx_\beta + d\vuu_\beta + s\vv_\beta), (\vv_\beta^\perp)^{k-1}\vv_\beta^{m-k} \right\rangle\,ds\\[2pt]
& \qquad + \sum_{\substack{j=1 \\ j\neq k}}^{m}\int_{0}^{\infty} \dfrac{d}{ds} \left\langle\left(
(\D^{\perp})^{\,m-j}\D^{\,j-1}\Psi^{(j)}\right)(\vx_\beta + d\vuu_\beta + s\vv_\beta), (\vv_\beta^\perp)^{k}\vv_\beta^{m-k-1} \right\rangle\,ds\\[2pt]
&= \left\langle \left(
(\D^{\perp})^{\,m-1}\Psi^{(0)}\right) (\vx_\beta +s\vuu_\beta), (\vuu_\beta^\perp)^{k-1}\vuu_\beta^{m-k} \right\rangle\Big|_{0}^{d}\\[2pt]
& \qquad +\sum_{\substack{j=1 \\ j\neq k}}^{m}\left\langle \left(
(\D^{\perp})^{\,m-j}\D^{\,j-1}\Psi^{(j)}\right) (\vx_\beta +s\vuu_\beta), (\vuu_\beta^\perp)^{k}\vuu_\beta^{m-k-1} \right\rangle\Big|_{0}^{d}\\[2pt]
& \qquad +  \left\langle\left(
(\D^{\perp})^{\,m-1}\Psi^{(0)}\right)(\vx_\beta + d\vuu_\beta + s\vv_\beta), (\vv_\beta^\perp)^{k-1}\vv_\beta^{m-k} \right\rangle\Big|_{0}^{\infty}\\[2pt]
& \qquad + \sum_{\substack{j=1 \\ j\neq k}}^{m} \left\langle\left(
(\D^{\perp})^{\,m-j}\D^{\,j-1}\Psi^{(j)}\right)(\vx_\beta + d\vuu_\beta + s\vv_\beta), (\vv_\beta^\perp)^{k}\vv_\beta^{m-k-1} \right\rangle\Big|_{0}^{\infty}.
\end{align*}
Therefore, we have
\begin{align}\label{eq; simplification of Lf}
 \Mc^{(k)}\vf (\beta,d)&= \left\langle \left(
(\D^{\perp})^{\,m-1}\Psi^{(0)}\right) (\vx_\beta +d\vuu_\beta), (\vuu_\beta^\perp)^{k-1}\vuu_\beta^{m-k} \right\rangle - \left\langle \left(
(\D^{\perp})^{\,m-1}\Psi^{(0)}\right) (\vx_\beta), (\vuu_\beta^\perp)^{k-1}\vuu_\beta^{m-k} \right\rangle\nonumber\\[2pt] & \qquad + \sum_{\substack{j=1 \\ j\neq k}}^{m}\left\langle \left(
(\D^{\perp})^{\,m-j}\D^{\,j-1}\Psi^{(j)}\right) (\vx_\beta +d\vuu_\beta), (\vuu_\beta^\perp)^{k}\vuu_\beta^{m-k-1} \right\rangle\nonumber\\[2pt]
& \qquad - \sum_{\substack{j=1 \\ j\neq k}}^{m}\left\langle \left(
(\D^{\perp})^{\,m-j}\D^{\,j-1}\Psi^{(j)}\right) (\vx_\beta), (\vuu_\beta^\perp)^{k}\vuu_\beta^{m-k-1} \right\rangle\nonumber\\[2pt]
& \qquad +  \lim_{s\to\infty}\left\langle\left(
(\D^{\perp})^{\,m-1}\Psi^{(0)}\right)(\vx_\beta + d\vuu_\beta + s\vv_\beta), (\vv_\beta^\perp)^{k-1}\vv_\beta^{m-k} \right\rangle \nonumber\\[2pt]
&\qquad -
\left\langle\left(
(\D^{\perp})^{\,m-1}\Psi^{(0)}\right)(\vx_\beta + d\vuu_\beta), (\vv_\beta^\perp)^{k-1}\vv_\beta^{m-k} \right\rangle\nonumber\\[2pt]
& \qquad + \lim_{s\to\infty}\sum_{\substack{j=1 \\ j\neq k}}^{m} \left\langle\left(
(\D^{\perp})^{\,m-j}\D^{\,j-1}\Psi^{(j)}\right)(\vx_\beta + d\vuu_\beta + s\vv_\beta), (\vv_\beta^\perp)^{k}\vv_\beta^{m-k-1} \right\rangle\nonumber\\[2pt]
& \qquad - \sum_{\substack{j=1 \\ j\neq k}}^{m} \left\langle\left(
(\D^{\perp})^{\,m-j}\D^{\,j-1}\Psi^{(j)}\right)(\vx_\beta + d\vuu_\beta), (\vv_\beta^\perp)^{k}\vv_\beta^{m-k-1} \right\rangle.
\end{align}
Using the boundary condition $\D^\ell\Psi^{(j)}\big|_{\partial\DR}=0,
\ \
\ell=0,\dots,m-1$ together with the fact that limiting values at $s$ tends to $\infty$ are vanishes in the relation \eqref{eq; simplification of Lf}, we conclude that $\Mc^{(k)}\vf = 0$. 

\noindent\textbf{Case 2: $k=0$.}
Here $\Mc^{(0)}\vf:=\Lc\vf$. Since the decomposition excludes the term $j=0$, each summand contains at least one $\D$.
Therefore, we have 
\begin{align}
    \vf
=
\sum_{j=1}^{m}
(\D^\perp)^{m-j}\D^j\Psi^{(j)}.
\end{align}
Using Lemma~\ref{Lemma: commute} together with the chain rule, each integrand appearing in the longitudinal V-line transform can be expressed as a total derivative along the corresponding ray direction. Proceeding as in the previous case, integration over the two segments of the V-line yields only boundary terms. Since the potentials $\Psi^{(j)}$ satisfy
\begin{align*}
    \D^\ell\Psi^{(j)}\big|_{\partial\DR}=0,
\qquad
\ell=0,\dots,m-1,
\end{align*}
all boundary contributions vanish. Consequently, we have 
$\Lc\vf=0$.

\noindent\textbf{Case 3: $k=m$.}
Here $\Mc^{(m)}\vf:=\Tc\vf$. Since the decomposition excludes the term $j=m$, each summand contains at least one $\D^\perp$. Therefore, we have
\begin{align*}
    \vf
=
\sum_{j=0}^{m-1}
(\D^\perp)^{m-j}\D^j\Psi^{(j)}.
\end{align*}
Using Lemma~\ref{Lemma: commute}, we may rewrite each term as
\begin{align*}
    (\D^\perp)^{m-j}\D^j\Psi^{(j)}
=
(\D^\perp)^{m-j-1}\D^j(\D^\perp\Psi^{(j)}).
\end{align*}
Hence, by applying the chain rule along the ray direction, we obtain
\begin{align*}
&\left\langle
(\D^\perp)^{m-j}\D^j\Psi^{(j)}
(\vx_\beta+s\vuu_\beta),
(\vuu_\beta^\perp)^m
\right\rangle  \\
&\qquad =
\frac{d}{ds}
\left\langle
(\D^\perp)^{m-j-1}\D^j\Psi^{(j)}
(\vx_\beta+s\vuu_\beta),
(\vuu_\beta^\perp)^{m-1}
\right\rangle .
\end{align*}
The same identity holds along the second ray with $\vuu_\beta$ replaced by $\vv_\beta$. Integrating over the two parts of the V-line and using the following boundary conditions of $\Psi^{(j)}$ 
\begin{align*}
    \D^\ell\Psi^{(j)}\big|_{\partial\DR}=0,
\qquad
\ell=0,\dots,m-1,
\end{align*}
we have 
$\Tc\vf=0$.


\noindent Hence, for every $0\leq k\leq m$, we obtain
\begin{align}\label{LF=0}
\vf = \sum_{\substack{j=0 \\ j\neq k}}^{m}
(\D^{\perp})^{\,m-j}\D^{\,j}\Psi^{(j)},
\qquad
\D^{\,\ell}\Psi^{(j)}\big|_{\partial\DR}=0,
\ \
\ell=0,\dots,m-1.
\quad
\Longrightarrow
\quad
\Mc^{(k)}\vf = 0. 
\end{align}
This completes the proof of the Theorem.    
\end{proof}

\subsection{Proof of Theorem~\ref{thm: special type tensor fields}}
 \begin{proof}
        From the proof of Theorem \ref{th:full data recovery} (see \cite[Section~3]{bhardwaj2025tensor}), for each $0\leq k\leq m,$ the mixed V-line transforms $ \Mc^{(k)}\vf(\beta,d)$ is related to the corresponding straight-line transforms $\mathcal{K}^{(k)}\vf $ through the following identity:
    \begin{align}\label{eq;relation btw V-line and st line transform for tensor field}
        \Kc^{(k)}\vf(\psi_{\beta},t_{d}) = \Mc^{(k)}\vf(\beta,d) + (-1)^{m-k}\Mc^{(k)}\vf(\beta + \pi,2R-d) - \Mc^{(k)}\vf({\beta,2R}),
    \end{align}
    where  $d\in[0,2R]$, $\beta \in [0, 2\pi)$, $\psi_{\beta} = \beta + \theta + \pi/2$ denotes the polar angle of the corresponding straight line, and  $t_{d} = (R-d)\sin( \theta)$ represents its signed distance from the origin.
Thus, it is enough to show how $\Psi^{(k)}$ can be recovered from
$\Kc^{(k)}\vf$.
 Next, we consider three cases.
   
 \noindent\textbf{Case 1: $k=0$.}  In this case, $\vf = 
(\D^{\perp})^{\,m}\Psi^{(0)}$, Using the definition of the straight-line longitudinal ray transform, we have
\begin{align*}
    \Kc^{(0)}\vf(\psi_{\beta},t_{d}) 
    &= \int_\mathbb{R} \left\langle (\D^{\perp})^{\,m}\Psi^{(0)} (t_{d} \vw_{\psi_{\beta}} + s\vw_{\psi_{\beta}}^\perp),(\vw_{\psi_{\beta}}^\perp)^{m}\right\rangle\,ds\\
    &= \int_\mathbb{R} \left\langle \D^m\Psi^{(0)} (t_{d} \vw_{\psi_{\beta}} + s\vw_{\psi_{\beta}}^\perp),\vw_{\psi_{\beta}}^{m}\right\rangle\,ds.
\end{align*}
Next, apply the chain rule repeatedly, and we get
\begin{align*}
    \Kc^{(0)}\vf(\psi_{\beta},t_{d}) 
    &= \int_\mathbb{R} \frac{d^m}{dt^m_d} \Psi^{(0)} (t_{d} \vw_{\psi_{\beta}} + s\vw_{\psi_{\beta}}^\perp)\,ds\\
    &=\frac{d^m}{dt^m_d}\mathcal{R}\Psi^{(0)}(\psi_{\beta},t_{d}).
\end{align*}
 After applying integration $n$-times, we have
\begin{align*}
    \mathcal{R}\Psi^{(0)}(\psi_{\beta},t_{d}) = \int\limits_{-\infty}^{t_d} \frac{(t_d -s)^{n-1}}{(n-1)!}\Kc^{(0)}\vf(\psi_{\beta},s)\,ds + \mathcal{P}_{n-1}(t_d),
\end{align*}
where $\mathcal{P}_{n-1}(t_d)$ is a polynomial of degree at most $n-1$.
Since $\Psi^{(0)}\in C_0^{\infty}(\DR)$ and using the boundary condition $\D^\ell\Psi^{(0)}\big|_{\partial\DR}=0,
\ \
\ell=0,\dots,m-1$, the polynomial terms must vanish, and thus we have
\begin{align*}
     \mathcal{R}\Psi^{(0)}(\psi_{\beta},t_{d}) = \frac{1}{(m-1)!}\int\limits_{-\infty}^{t_d} (t_d -\tau)^{m-1}\Kc^{(0)}\vf(\psi_{\beta},\tau)\,d\tau.
\end{align*}
Finally, applying the inversion formula of the Radon transform \eqref{eq:Radon's inversion formula} and recover the scalar function $\Psi^{(0)}$. Consequently, we have the symmetric $m$-tensor field $\vf$ is recovered explicitly from $\Mc^{(0)}\vf$.

\noindent\textbf{Case 2: $k=m$.} In this case,
$\vf=\D^m\Psi^{(m)}$.
Using the definition of the straight-line transverse ray transform $\Kc^{(m)}\vf$, we obtain
\begin{align*}
\Kc^{(m)}\vf(\psi_\beta,t_d)
&=
\int_{\mathbb{R}}
\left\langle
\D^m\Psi^{(m)}
(t_d\vw_{\psi_\beta}+s\vw_{\psi_\beta}^{\perp}),
\vw_{\psi_\beta}^{m}
\right\rangle ds \\
&=
\int_{\mathbb{R}}
\frac{d^m}{dt^m_d}
\Psi^{(m)}
(t_d\vw_{\psi_\beta}+s\vw_{\psi_\beta}^{\perp})\,ds \\
&=
\frac{d^m}{dt^m_d}
\mathcal{R}\Psi^{(m)}(\psi_\beta,t_d).
\end{align*}
Therefore, we have
\begin{align}
   \frac{d^m}{dt^m_d}
\mathcal{R}\Psi^{(m)}(\psi_\beta,t_d)
=
\Kc^{(m)}\vf(\psi_\beta,t_d).
\end{align}
Since $\Psi^{(m)}\in C_0^\infty(\DR)$ and the boundary condition $\D^\ell\Psi^{(m)}\big|_{\partial\DR}=0,
\ \
\ell=0,\dots,m-1$, its Radon transform is given by
\begin{align*}
    \mathcal{R}\Psi^{(m)}(\psi_\beta,t_d)
=
\frac{1}{(m-1)!}
\int_{-\infty}^{t_d}
(t_d-\tau)^{m-1}
\Kc^{(m)}\vf(\psi_\beta,\tau)\,d\tau.
\end{align*}
Finally, applying the inversion formula of the Radon transform \eqref{eq:Radon's inversion formula} and recover the scalar function $\Psi^{(m)}$. Consequently, we have the symmetric $m$-tensor field
$\vf=\D^m\Psi^{(m)}$
is recovered explicitly from $\Mc^{(m)}\vf$.

\noindent\textbf{Case 3: $1\leq k \leq m-1$.} In this case, $\vf = 
(\D^{\perp})^{\,m-k}\D^{\,k}\Psi^{(k)}$. Using the definition of the straight-line $k$-th mixed ray transform, we have
\begin{align*}
    \Kc^{(k)}\vf(\psi_{\beta},t_{d}) 
    &= \int_\mathbb{R} \left\langle (\D^{\perp})^{\,m-k}\D^{\,k}\Psi^{(k)} (t_{d} \vw_{\psi_{\beta}} + s\vw_{\psi_{\beta}}^\perp),\vw_{\psi_{\beta}}^k(\vw_{\psi_{\beta}}^\perp)^{m-k}\right\rangle\,ds.
\end{align*}
Using the definition of the symmetric tensor product and applying repeated application of the chain rule gives
\begin{align}
   \left\langle
(\D^\perp)^{m-k}\D^k\Psi^{(k)}(x),
\vw_{\psi_\beta}^{k}(\vw_{\psi_\beta}^\perp)^{m-k}
\right\rangle
=
C_{m,k}\,
\partial_{t_d}^{\,m}\Psi^{(k)}(x), 
\end{align}
where
\begin{align}
 C_{m,k}=\binom{m}{k}^{-1}.   
\end{align}
Therefore, we have
\begin{align*}
\Kc^{(k)}\vf(\psi_\beta,t_d)
&=
C_{m,k}
\int_{\mathbb R}
\frac{d^m}{dt^m_d}
\Psi^{(k)}
(t_d\vw_{\psi_\beta}+s\vw_{\psi_\beta}^\perp)\,ds \\
&=
C_{m,k}
\frac{d^m}{dt^m_d}
\mathcal{R}\Psi^{(k)}(\psi_\beta,t_d).
\end{align*}
Thus, we obtain 
\begin{align}
 \frac{d^m}{dt^m_d}
\mathcal{R}\Psi^{(k)}(\psi_\beta,t_d)
=
\frac{1}{C_{m,k}}
\Kc^{(k)}\vf(\psi_\beta,t_d).   
\end{align}
Since $\Psi^{(k)}\in C_0^\infty(\DR)$ and the boundary condition $\D^\ell\Psi^{(k)}\big|_{\partial\DR}=0,
\ \
\ell=0,\dots,m-1$, its Radon transform is given by
\begin{align*}
    \mathcal{R}\Psi^{(k)}(\psi_\beta,t_d)
=
\frac{C_{m,k}}{(m-1)!}
\int_{-\infty}^{t_d}
(t_d-\tau)^{m-1}
\Kc^{(k)}\vf(\psi_\beta,\tau)\,d\tau.
\end{align*}
Finally, applying the inversion formula of the Radon transform \eqref{eq:Radon's inversion formula} and recover the scalar function $\Psi^{(k)}$. Consequently, we have the symmetric $m$-tensor field
$\vf=(\D^{\perp})^{\,m-k}\D^{\,k}\Psi^{(k)}$
is recovered explicitly from $\Mc^{(k)}\vf$. This completes the proof.
\end{proof}
\section{Numerical Implementation}\label{sec: Numerical implementation}
This section is devoted entirely to the numerical implementation of the proposed method and to evaluating its performance on different phantoms. We present the numerical reconstruction only for vector fields $(m=1)$ and symmetric 2-tensor fields $(m=2)$. To avoid repetition, we have used different combinations of data for different phantoms and vector/tensor fields. Because of this, our numerical validation is primarily divided into two parts: 
\begin{enumerate}
\item First, we generate the forward data for a given field $\vf$ over the domain of $[-1, 1]\times [-1, 1]$. For the vector field case, we generate V-line data, longitudinal V-line transform $\mathcal{L}\vf$ and transverse V-line transform $\mathcal{T}\vf$, and for the tensor field case, we generate V-line data, longitudinal V-line transform $\mathcal{L}\vf$, transverse V-line transform $\mathcal{T}\vf$ and mixed V-line transform $\mathcal{M}\vf$. 
\item Utilizing the data generated in the first step, we reconstruct the underlying components of the vector/tensor fields. More specifically, we numerically reconstruct the vector field $\vf = (f_1, f_2)$ using the data obtained from the V-line data, $\mathcal{L}\vf$ and $\mathcal{T}\vf$.  Analogously, for symmetric $2$-tensor field $\vf  = \begin{pmatrix}
    f_{11} & f_{12}\\
    f_{21} & f_{22}
\end{pmatrix} $, we use the information of V-line  data $\mathcal{L}\vf$, $\mathcal{T}\vf$, and $\mathcal{M}\vf$. In this case, we will identify $\vf$ with $(f_{11}, f_{12}, f_{22})$, since $\vf$ is symmetric. 
\end{enumerate}
A more detailed discussion of both the data generation process and the subsequent numerical reconstruction begins with the vector field case in Subsection \ref{Sec: Reconstruction of vector field}, and then the tensor field case in Subsection \ref{Subsec: Rec Tensor}.

\subsection{Reconstruction of Vector Fields}\label{Sec: Reconstruction of vector field}
\subsubsection{Data Generation}\label{Sec: Data Formation}
To generate data for the numerical simulations, we consider a circular domain of radius $R=1$ and fix the scattering angle at $\theta = \pi/4$. To satisfy the theoretical requirement that the field be supported within a disk of radius $R \sin \theta$, a spatial mask is applied to the considered phantom. We compute the discrete forward data by evaluating the longitudinal ($\mathcal{L}\vf$) and transverse ($\mathcal{T}\vf$) V-line transforms over a grid of source angles $\beta$ and scattering depths $d$. 
For each pair $(\beta, d)$, the numerical integration is performed in two stages:
\begin{enumerate}
    \item \textbf{Incident Segment:} The ray is traced from the boundary source to the vertex at depth $d$ along the direction $\vuu$. Since discrete points along this path rarely align with the Cartesian grid, 2D bilinear interpolation is used to sample the vector field components $(f_1, f_2)$. We then compute the inner products of $(f_1, f_2)$ with $\vuu$ and $\vuu^{\perp}$. 
    \item \textbf{Scattered Segment:} The path continues from the vertex along the deflected direction $\vv$. We follow the same interpolation and integration procedure to compute the contribution coming from this branch.
\end{enumerate}
The line integrals are approximated using the composite trapezoidal rule. Finally, the total discrete measurements for $\mathcal{L}\vf(\beta,d)$ and $\mathcal{T}\vf(\beta,d)$ are obtained by summing the contributions from both segments.


\subsubsection{Numerical Reconstruction} 
\label{Sec: Numerical Reconstruction (Vector Field)}

In Figures \ref{Fig: Smooth Gaussian without Noise}--\ref{fig:Different_thetas_Ph3}, we demonstrate the recovery of the scalar functions $f_1$ and $f_2$, components of a vector field $\vf$ across various phantoms, both with noise and without noise. 
To achieve this, we first utilize our V-line transform data $\Lc\vf$ and $\Tc\vf$ to generate longitudinal and transverse ray transforms, $\Ic\vf$ and $\Jc\vf$ (straight line transforms), via the following relation (for more details see \cite[Section~4]{bhardwaj_2024}):
\begin{align}\label{relation btw V-line and st line}
    \begin{cases}
        \Ic\vf(\psi_{\beta},t_{d}) = \Lc\vf(\beta,d) - \Lc\vf(\beta + \pi,2R-d) - \Lc \vf(\beta, 2R),\\[5pt]
        \Jc\vf(\psi_{\beta},t_{d}) = \Tc\vf(\beta,d) - \Tc\vf(\beta + \pi,2R-d) - \Tc\vf(\beta, 2R),
    \end{cases}
\end{align}
where $d \in [0,2R]$, $\beta \in [0, 2\pi)$, $\psi_{\beta} = \beta + \theta + \pi/2$, and $t_{d} = (R-d)\sin(\pi + \theta)$. 
\vspace{4mm}\\
\noindent Recall, for $\psi \in [0, 2\pi), \ p \in \mathbb{R},$ the longitudinal ray transform $\Ic\vf$ and transverse ray transform $\Jc\vf$ are given by:
\begin{equation}\label{eq:def LRT}
\mathcal{I}\vf (\psi,p) = \mathcal{I}\vf (\vw,p) := \int_\mathbb{R} \vw^\perp \cdot \vf (p \vw + s\vw^\perp)\,ds,
\end{equation}
\begin{equation}\label{eq:def TRT}
\mathcal{J}\vf (\psi,p) = \mathcal{J}\vf (\vw,p) := \int_\mathbb{R} \vw \cdot \vf (p \vw + s\vw^\perp)\,ds.
\end{equation}
These transforms can be cast in matrix form to isolate the Radon transforms of the individual scalar components:
\begin{equation}\label{eq:matrix_system}
   \begin{bmatrix}
       \mathcal{I}\vf \\
       \mathcal{J}\vf
   \end{bmatrix}(\vw, p)
   = 
     \begin{bmatrix}
         -w_2 & w_1 \\
         w_1  & w_2
     \end{bmatrix}
     \begin{bmatrix}
         \mathcal{R}f_1 \\
         \mathcal{R}f_2
     \end{bmatrix}(\vw, p), \quad \mbox{ here } \vw = (w_1, w_2) = (\cos \psi, \sin \psi).
\end{equation}
Because the transformation matrix in \eqref{eq:matrix_system} is invertible, we can point-wise solve this system to isolate the component-wise Radon transforms, $\mathcal{R}f_1$ and $\mathcal{R}f_2$. Finally, applying the inverse Radon transform yields the reconstructed scalar components $f_1$ and $f_2$. 

\noindent We present numerical experiments starting with a baseline, noise-free case. We evaluate the performance of our inversion algorithm using three distinct phantom classes defined on the domain $[-1, 1] \times [-1, 1]$. Each class introduces varying degrees of mathematical regularity and geometric complexity: smooth bump functions (Phantom 1 - PH1), characteristic functions of overlapping disks (Phantom 2 - PH2), and characteristic functions of rectangular annuli (Phantom 3 - PH3).
As a first test case, we define the components $f_1$ and $f_2$ using smooth Gaussian bump functions centered at different locations with distinct supports. The mathematical function for each bump is defined as:
\begin{equation}
    f_i(x,y) = \begin{cases}
    e^{-s^2 / \left(s^2 - \left[(x - a)^2+(y - b)^2\right]\right)}, & (x - a)^2+(y - b)^2 < s^2, \\[2pt]
    0, & (x - a)^2+(y - b)^2 \geq s^2,
    \end{cases}
\end{equation}
where $(a, b)$ denotes the center of the bump function, and $s$ represents the radius of its support. The selected parameters are summarized in Table \ref{tab:single_bump_params}. We will refer to this phantom as PH1. Later, for the tensor field, there will be one more component, which we will also call PH1.

\begin{table}[H]
    \centering
    \begin{tabular}{lccc}
        \toprule
        \textbf{Component} & \textbf{Center} $(a, b)$ & \textbf{Support Radius} ($s$) \\
        \midrule
        $f_1$  & $(0.00, 0.00)$ & $0.15$ \\
        $f_2$  & $(0.40, 0.40)$ & $0.10$ \\ 
        \bottomrule
    \end{tabular}
    \caption{Center coordinates and support radii for the components of the vector field $\vf$.}
    \label{tab:single_bump_params}
\end{table}
\noindent Figure \ref{Fig: Smooth Gaussian without Noise} displays the results for this smooth test case. The first column shows the ground-truth scalar components $f_1$ and $f_2$. The second column illustrates the simulated forward data $\Lc\vf(\beta, d)$ and $\Tc\vf(\beta,d)$ plotted against the angular variable $\beta$ (in degrees) and the distance parameter $d$. The final column displays the reconstructed components obtained via the proposed inversion method. We see that the reconstructed scalar components are really good in the absence of noise (see Table \ref{table:2} for the relative error).
\begin{figure}[H] 
\centering
{\includegraphics[width=0.85 \textwidth]{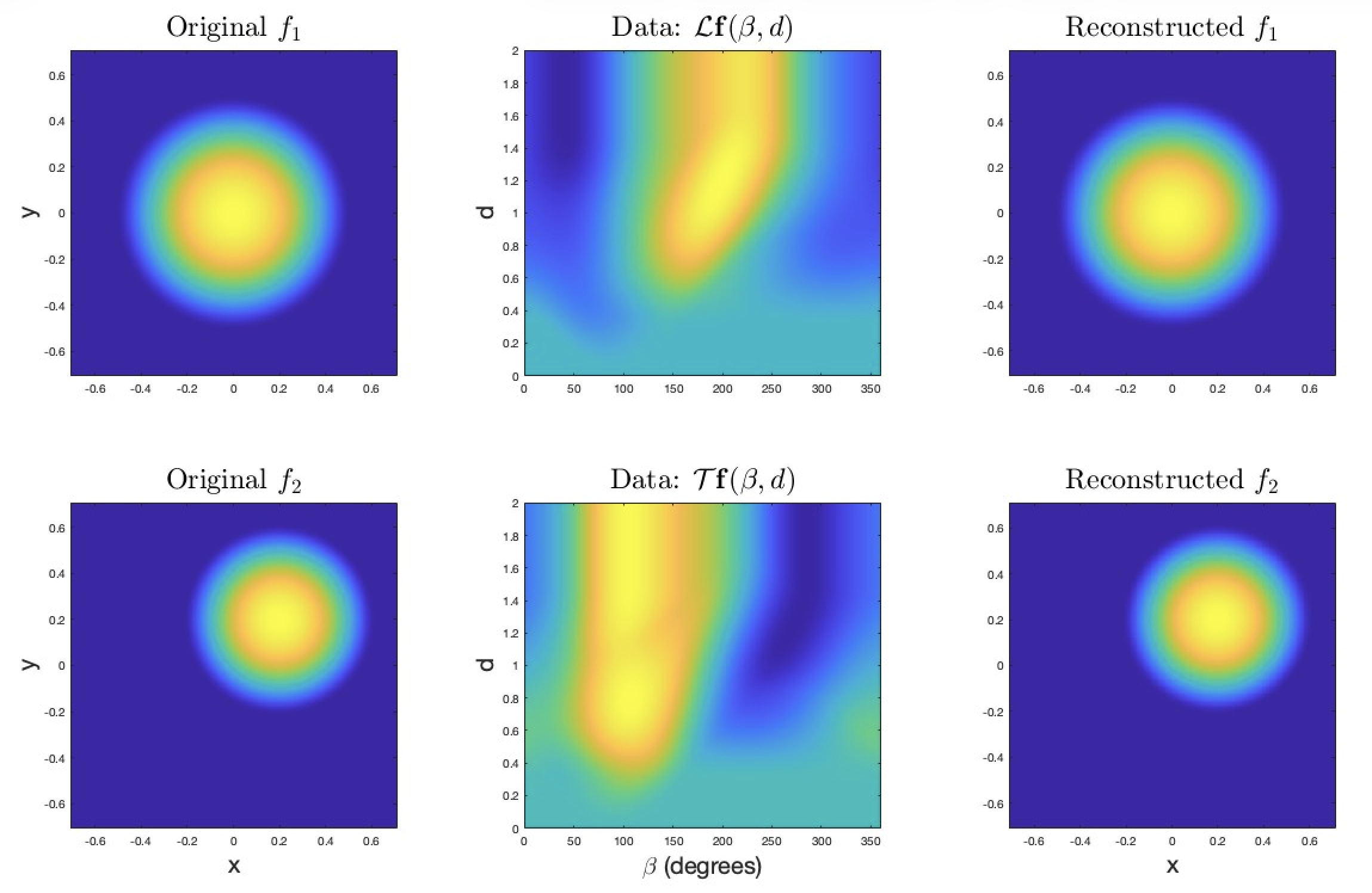}}
    \caption{Original components of the vector field \(\vf\) (first column), the associated V-line transform data \(\Lc\vf\) and \(\Tc\vf\) (second column), and the reconstructed vector field components (third column).}
    \label{Fig: Smooth Gaussian without Noise}
  \end{figure}
\noindent To investigate the stability of the proposed inversion method, we next introduce varying levels of Gaussian noise to the forward projection data. Figure \ref{fig; Reconstructed components of f with different levels of noise} illustrates the reconstruction of phantom 1 under noise levels of $5\%, 10\%$ and $20\%$. In the reconstruction, we observe that increasing the noise level introduces visible granular artifacts into the spatial domain. However, we observe that the locations, shapes, and amplitudes of $f_1$ and $f_2$ are well reconstructed. Thus, the reconstruction remains robust in the presence of noise.
\begin{figure}[H] 
\centering
{\includegraphics[width=0.9\textwidth]{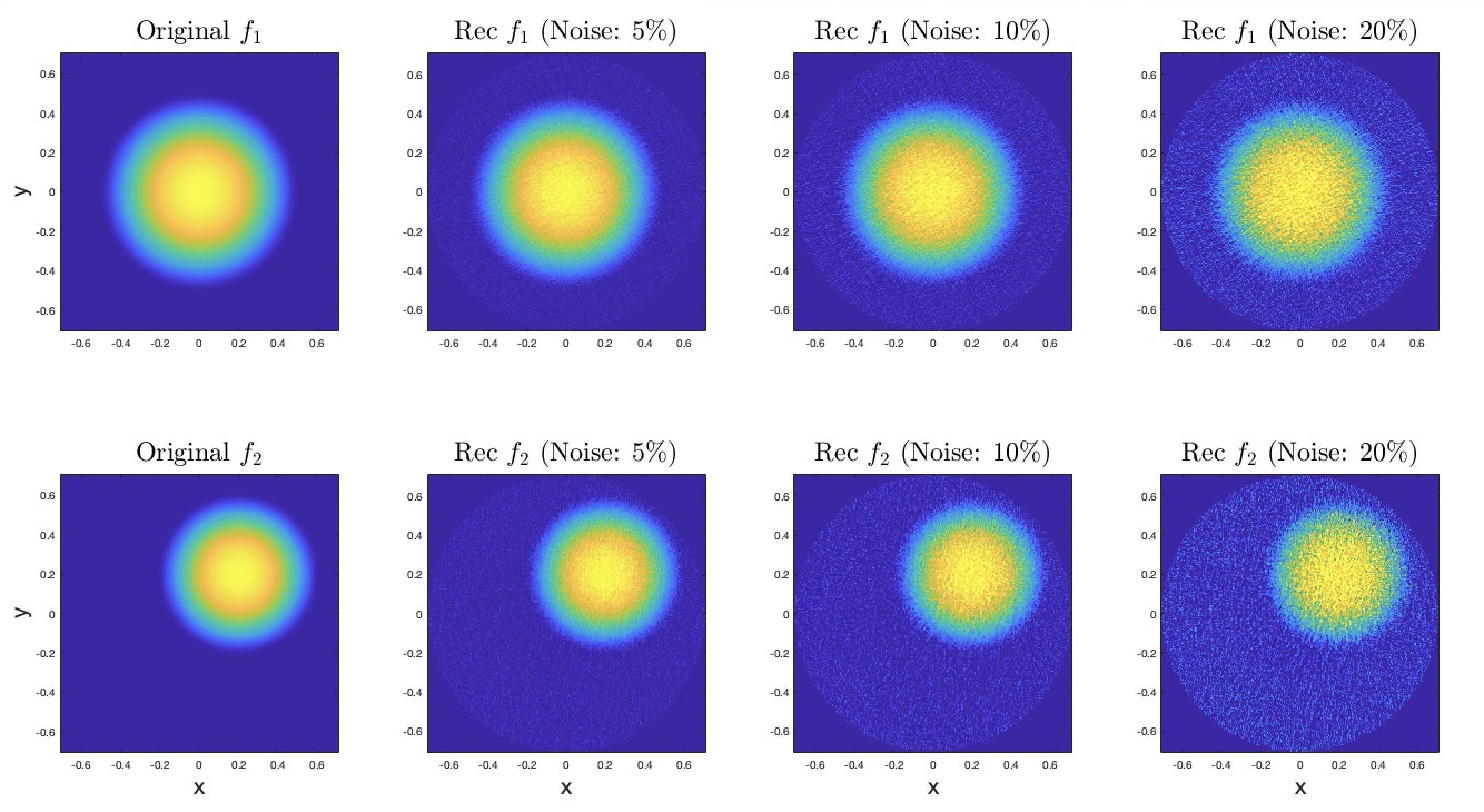}}
    \caption{Reconstructed components of $\vf$ with different levels of noise.}\label{fig; Reconstructed components of f with different levels of noise}
  \end{figure}

 \noindent We now consider another phantom, which we refer to as Ph2. In this class, each component of the phantom is represented as a weighted combination of three characteristic functions of disks having different radii $r$ and centers $(a,b)$. The values of these parameters are given in Table~\ref{tab:compact_multidisk_params} : 
\begin{table}[htbp]
    \centering
    \begin{tabular}{@{}lccc@{}}
        \toprule
        \textbf{Components} & \textbf{Centers ($a, b$)} & \textbf{Radii ($r$)} & \textbf{Intensities ($val$)} \\[4pt]
        \midrule
        $f_1$ & $(-0.2, 0.1)$, $(0.15, 0.15)$, $(0.0, -0.15)$ & $0.25, 0.3, 0.3$ & $3, 3.5, 4$ \\[4pt]
        $f_2$ & $(0, 0.2)$, $(0.2, -0.1)$, $(-0.2, -0.15)$ & $0.2, 0.25, 0.35$ & $2, 5, 1.5$ \\[4pt]
        \bottomrule
    \end{tabular}
      \caption{Parameter sets are ordered sequentially corresponding to Disks $[1, 2, 3]$.}
      \label{tab:compact_multidisk_params}
\end{table}

\noindent From Figure \ref{fig: overlapping circle}, it is evident that the method handles the sharp discontinuities of the phantom with minimal visual artifacts.
\begin{figure}[H] 
\centering
{\includegraphics[width=0.7\textwidth]{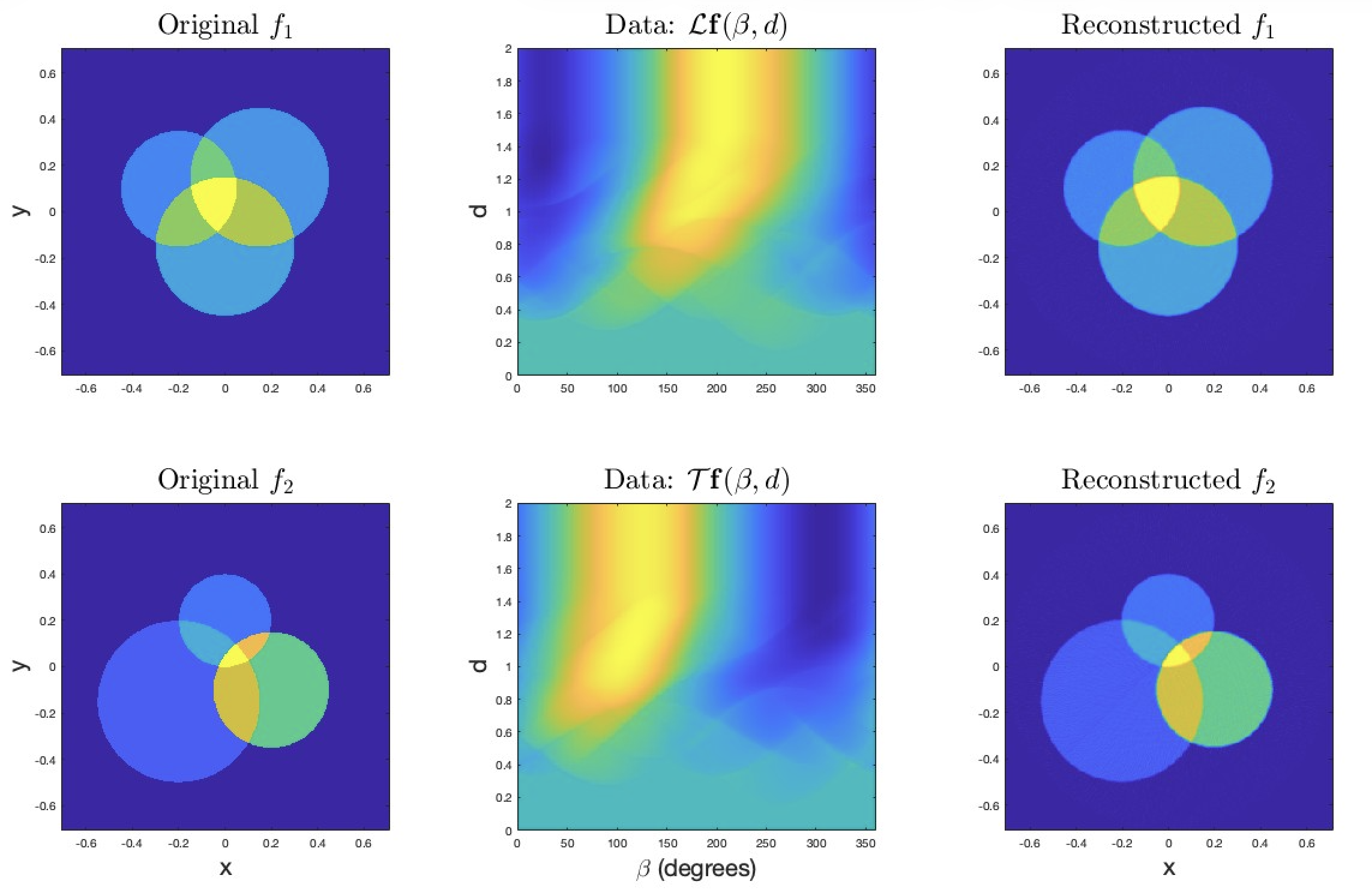}}
    \caption{Original components of the vector field \(\vf\) (first column), the associated V-line transform data \(\Lc\vf\) and \(\Tc\vf\) (second column), and the reconstructed vector field components (third column).}
    \label{fig: overlapping circle}
  \end{figure}
  
 \noindent Extending the validation of the inversion scheme, Figure \ref{fig: Reconstructed components of f with different levels of noise} evaluates the numerical stability of the algorithm against data perturbation by reconstructing the piecewise constant field from noisy projection. As in the previous cases, the leftmost column shows the original components ($f_1$ and $f_2$), while the subsequent columns show the reconstructed fields after corrupting the transform data with $5\%, 10\%, $ and $20\%$ noise, respectively. We observe that reconstructed components have prominent granular variations and distinct circular boundary artifacts. But the numerical implementation also shows robustness, as it can capture the fundamental features even under severe $20 \%$ noise conditions. 
\begin{figure}[H] 
\centering
{\includegraphics[width=0.8\textwidth]{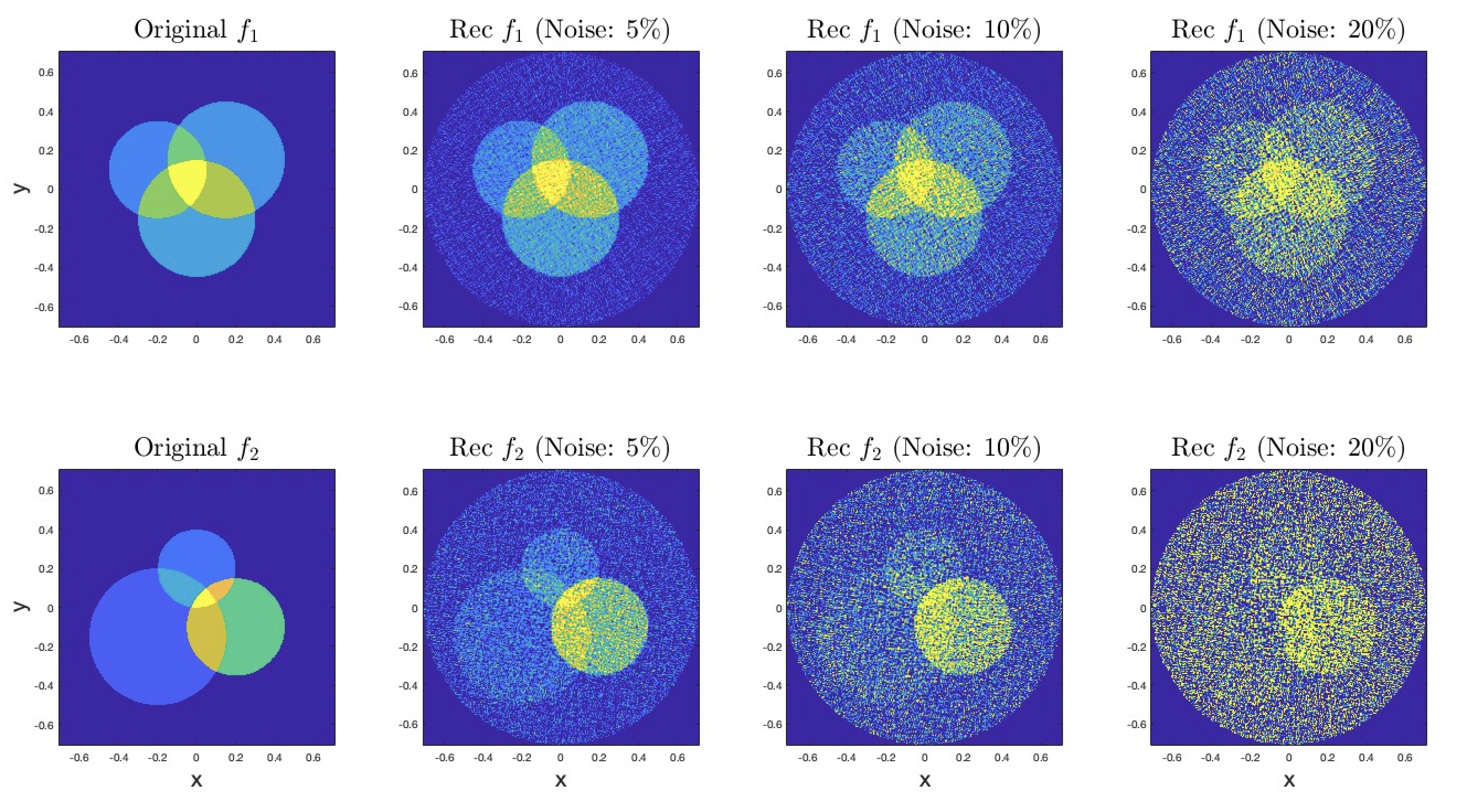}}
    \caption{Reconstructed components of $\vf$ with different levels of noise.}
    \label{fig: Reconstructed components of f with different levels of noise}
  \end{figure}

\noindent Further expanding on the validation of the inversion framework, we consider a vector field characterized by non-convex, piecewise-constant geometries, specifically rotated square-ring structures. 
This last class of phantoms consists of square annuli with identical dimensions, centered at different locations, and rotated at different angles with respect to the $x$-axis. We will refer to this phantom as Ph3. For details about the locations and rotations, please refer to the Table \ref{tab:phantom_params}. 
\begin{table}[htbp]
    \centering
    \begin{tabular}{@{}lccc@{}}
        \toprule
        \textbf{Components} & \textbf{Center ($a, b$)} & \textbf{Rotation angle} & \textbf{Intensity ($val$)} \\
        \midrule
        $f_1$ & $(0.00, 0.00)$ & $10^\circ$ & $1.00$ \\
        $f_2$ & $(-0.10, -0.10)$ & $45^\circ$ & $0.95$ \\
        \bottomrule
    \end{tabular}    \caption{Phantom Position and Orientation Parameters}
    \label{tab:phantom_params}
\end{table}

\noindent Table \ref{table:2} summarizes the relative errors in reconstructing $f_1$ and $f_2$ from V-line data $\mathcal{L} \vf$ and $\mathcal{T}\vf$ for the three phantoms (PH1--PH3) under different noise levels. The results show that the reconstruction error increases as the noise level rises.  Figure \ref{fig: square} demonstrates the algorithm's ability to accurately reconstruct the components of a vector field defined on the phantom, which have sharp corners and hollow interiors. 

\begin{figure}[H] 
\centering
{\includegraphics[width=0.8\textwidth]{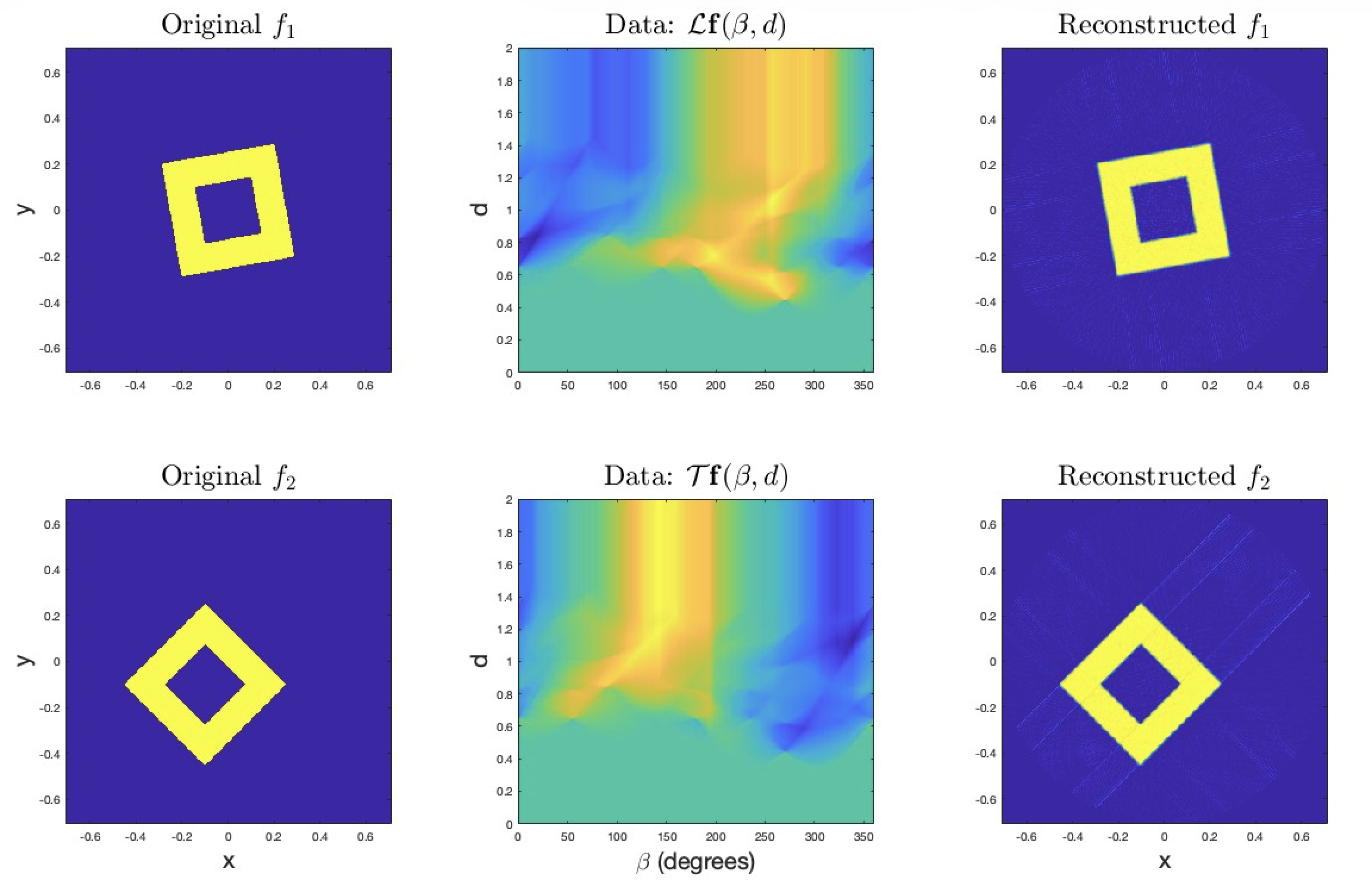}}
\caption{Original components of the vector field $\vf$ (first column), the associated V-line transform data $\Lc\vf$ and $\Tc\vf$ (second column), and the reconstructed vector field components (third column).}
\label{fig: square}
\end{figure}

\noindent Further, Figure \ref{fig: square with noise} shows the successful reconstruction of the non-convex square ring phantom. This figure evaluates the algorithm's stability by introducing $5 \%, 10 \%$ and $20 \%$ noise to the projected data. 
  
\begin{figure}[H] 
\centering
{\includegraphics[width=0.85\textwidth]{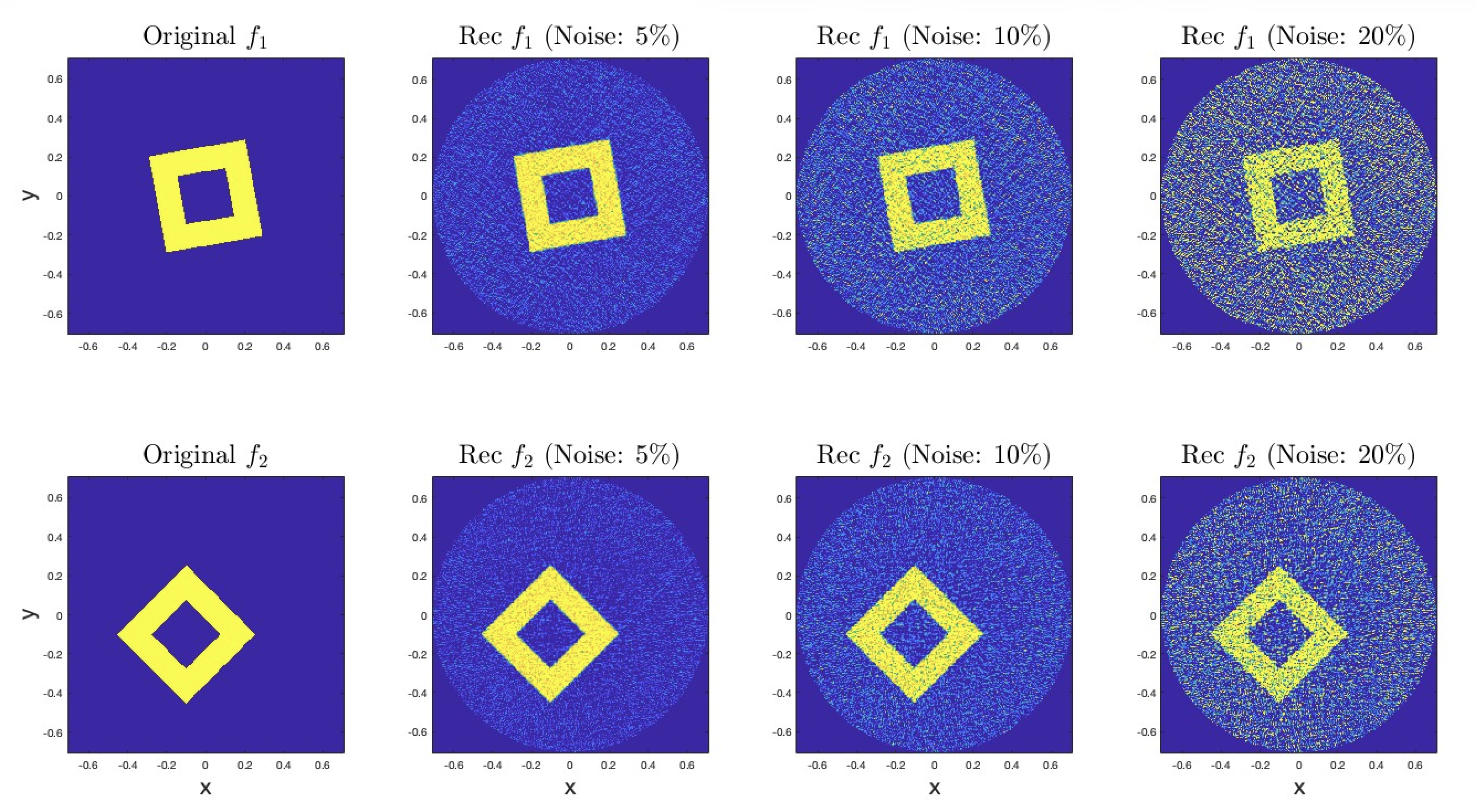}}
\caption{Reconstructed components of $\vf$ with different levels of noise.}
\label{fig: square with noise}
\end{figure}

\begin{table}[h!]
\centering 
\begin{tabular}{ |p{1.6 cm}||p{0.6cm}||p{1.6cm}|p{1.6cm}|p{1.7cm}|p{1.7cm}|  }
 \hline
Phantoms & $\vf$ & No noise & 5\% noise  & 10\% noise & 20\% noise\\[3pt]
 \hline
PH1   & $f_1$   & 1.29\% &	5.33\%	&  10.40\%	& 20.82\%\\[3pt]
 \hline
PH1   &  $f_2$  & 1.57\% &	8.38\%	& 16.46\%	& 33.04\%\\[3pt]
 \hline
 PH2  &  $f_1$  & 8.29\%	& 31.72\%	& 62.25\%	& 123.26\%\\[3pt]
  \hline
 PH2    & $f_2$ & 9.52\%	& 60.84\%	& 120.30\%	& 239.77\%\\[3pt]
 \hline 
 PH3  &  $f_1$ & 15.62\%	& 39.04\%	& 72.89\%	& 143.92\%\\[3pt]
  \hline
 PH3    & $f_2$ & 16.37\% &	31.26\%	& 55.52\%	& 106.95\%\\[3pt]
 \hline 
\end{tabular}
 \caption{Relative errors in the reconstructions of $f_1$ and $f_2$ from $\Lc\vf$ and $\Tc\vf$.}
\label{table:2}
\end{table}
\noindent 
Table \ref{table:2} summarizes the relative errors in reconstructing $f_1$ and $f_2$ from V-line data $\mathcal{L} \vf$ and $\mathcal{T}\vf$ for the three phantoms (PH1--PH3) under different noise levels. The results show that the reconstruction error increases as the noise level rises. Next, we present a brief comparison of reconstructions from different breaking angles $\theta$. Recall that this angle $\theta$ directly dictates the observable field of view (FOV), as seen from the theoretical results  (Theorem~\ref{th:full data recovery}) that the reconstruction works if the unknown field is supported in a disk of radius $R \sin(\theta)$. Consequently, small angles (e.g., $15^\circ$) yield a severely restricted FOV; any part of the vector field outside this small central region suffers from severe truncation artifacts. In contrast, as the opening angle widens (e.g., $75^\circ$), the valid mathematical FOV expands to encompass nearly the entire domain, allowing full reconstruction. We have presented this comparison for Phantom 2 (see Figure \ref{fig:Different_thetas_Ph2}) and Phantom 3 (see Figure \ref{fig:Different_thetas_Ph3}) with $\theta$ taking the values $15^\circ$, $30^\circ$, $45^\circ$, $60^\circ,$ and $75^\circ$. 
  \begin{figure}[H] 
\centering
{\includegraphics[width=1 \textwidth]{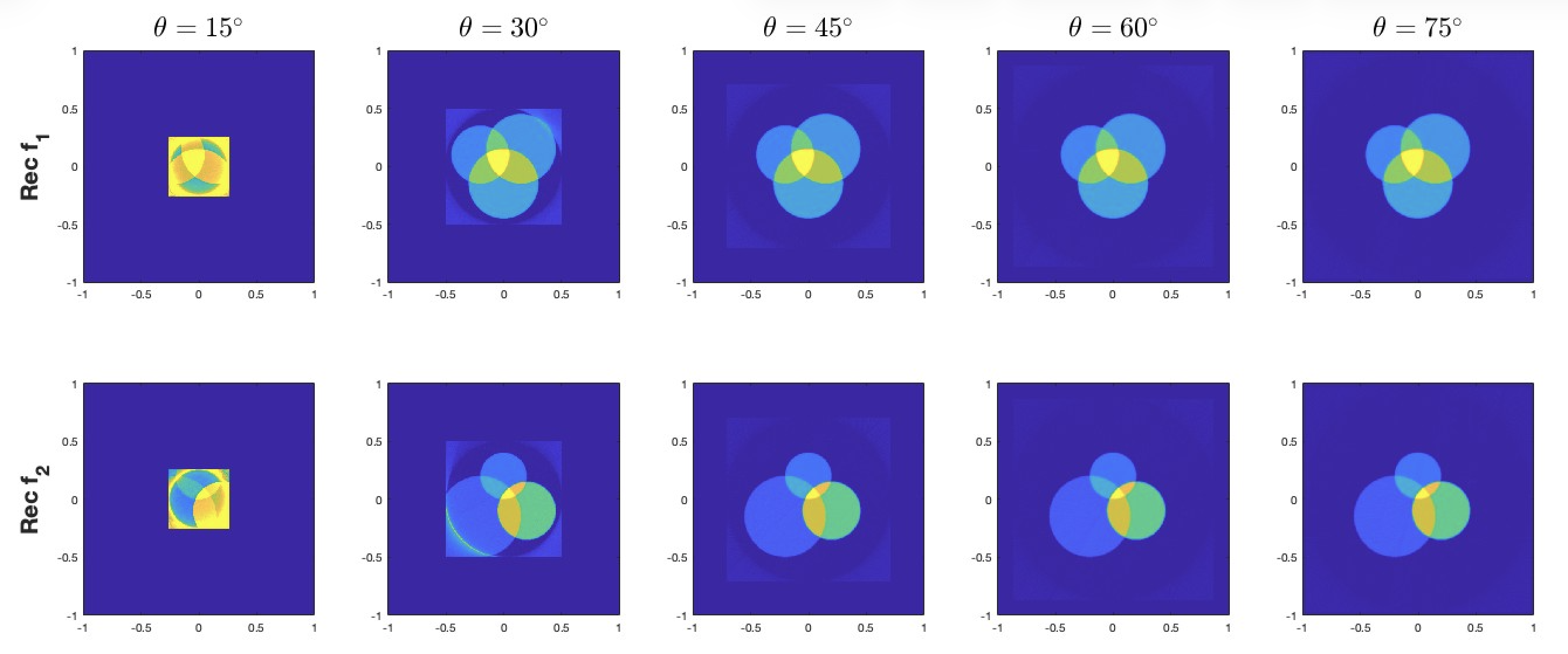}}
    \caption{Reconstructed components of $\vf$ (Phantom 2) for different angles ($\theta = 15^\circ$ to $75^\circ$).}\label{fig:Different_thetas_Ph2}
  \end{figure}
  \begin{figure}[H] 
\centering
{\includegraphics[width=1 \textwidth]{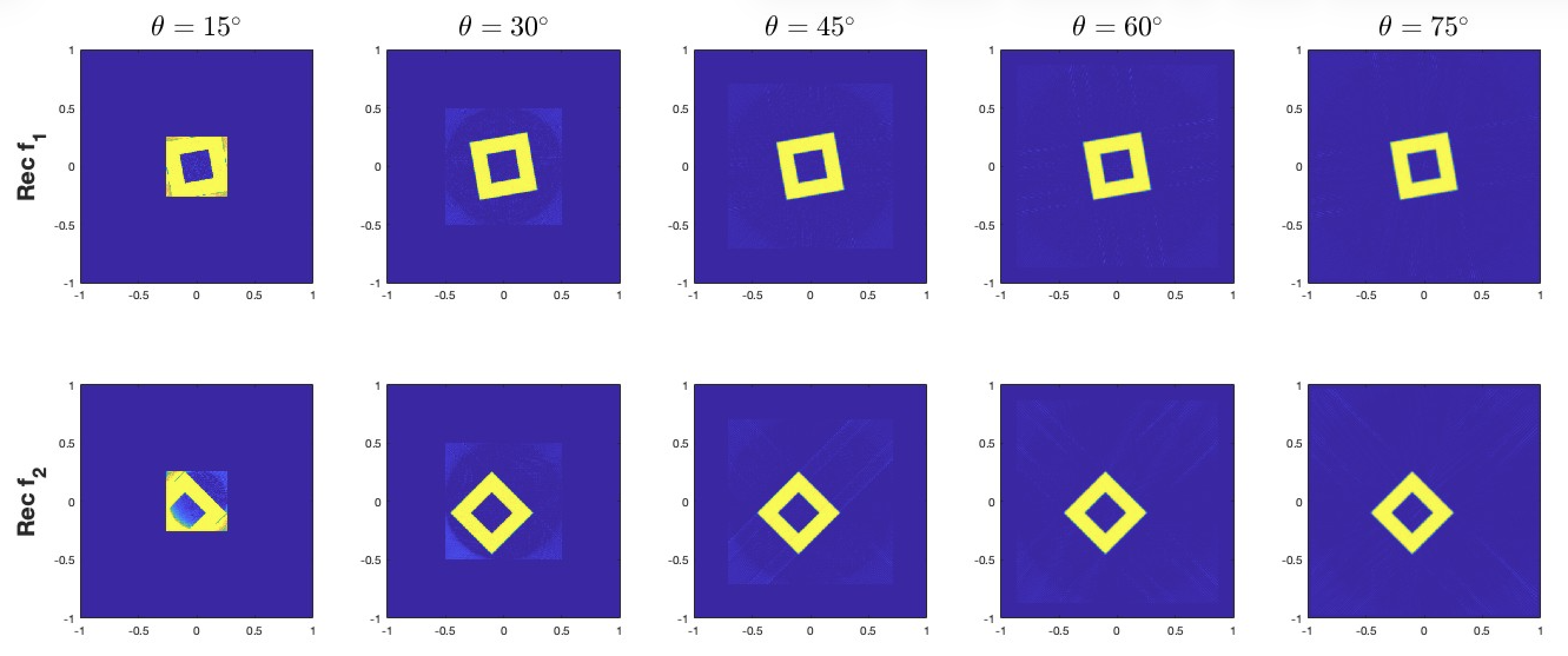}}
    \caption{Reconstructed components of $\vf$ (Phantom 3) for different angles ($\theta = 15^\circ$ to $75^\circ$).}\label{fig:Different_thetas_Ph3}
  \end{figure}  
\subsection{Reconstruction of Tensor Fields}\label{Subsec: Rec Tensor}

To generate the data, we will follow the same procedure as discussed in subsection \ref{Sec: Reconstruction of vector field}. The only difference is that we will now consider the three components of $\vf =(f_{11}, f_{12},f_{22})$, and we compute $\Lc\vf$, $\Mc\vf$, and $\Tc\vf$. Now the goal is to recover $f_{11}, f_{12}$ and $f_{22}$  from the given $\Lc\vf$, $\Mc\vf$, and $\Tc\vf$.
Figures [\ref{fig:Rec_tensor_without_noise_Ph2}-\ref{fig:Rec_tensor_with_noise_Ph3}], demonstrates the recovery of the components $f_{11}$, $f_{12}$, and $f_{22}$ of a tensor field $\vf$ across various phantoms, both with and without noise. 
To achieve this, we follow a similar approach analogous to the vector field case, i.e., we first compute the following formulas to get $\Ic\vf$, $\Kc\vf$, and $\Jc\vf$ using $\Lc\vf$, $\Mc\vf$, and $\Tc\vf$:
\begin{align*}\label{eq;:relation btw V-line and st line transform for tensor field}
        \begin{cases}
            \Ic\vf{{(\psi_{\beta},t_{d})}} = \Lc\vf(\beta,d) +  \Lc\vf(\beta + \pi,2R-d) - \Lc \vf(\beta, 2R),\\[4pt]
            \Kc\vf(\psi_{\beta},t_{d}) = \Mc\vf(\beta,d) -\Mc\vf(\beta + \pi,2R-d) - \Mc\vf({\beta,2R}),\\[4pt]
            \Jc\vf(\psi_{\beta},t_{d}) = \Tc\vf(\beta,d) +  \Tc\vf(\beta + \pi,2R-d) - \Tc\vf(\beta, 2R),
        \end{cases}
    \end{align*}
    where  $d\in[0,2R]$, $\beta \in [0, 2\pi)$, $\psi_{\beta} = \beta + \theta + \pi/2 \ \mbox { and } \  t_{d} = (R-d)\sin(\pi + \theta)$.
For $\psi \in [0, 2\pi) \mbox { and } p \in \mathbb{R}$, we recall the longitudinal ray transform $\mathcal{I}\vf$, mixed ray transform $\mathcal{K}\vf$ and transverse ray transform $\mathcal{J}\vf$, which are defined by
\begin{equation*}\label{eq:def LRT for tensor}
\mathcal{I}\vf (\psi,p) = \mathcal{I}\vf (\vw,p) := \int_\mathbb{R} \left\langle\vf (p \vw + s\vw^\perp), (\vw^\perp)^{2} \right\rangle\,ds,
\end{equation*}
\begin{equation*}\label{eq:def MRT}
\mathcal{K}\vf (\psi,p) = \mathcal{K}\vf (\vw,p) := \int_\mathbb{R} \left\langle \vf (p \vw + s\vw^\perp), \vw^{\perp}\odot\vw\right\rangle\,ds,
\end{equation*} 
\begin{equation*}\label{eq:def TVT for tensor}
\mathcal{J}\vf (\psi,p) = \mathcal{J}\vf (\vw,p) := \int_\mathbb{R} \left\langle \vf (p \vw + s\vw^\perp), \vw^{2} \right\rangle\,ds,
\end{equation*}
respectively.
Observe that, 
\vspace{-5 mm}
\begin{equation*}
    \begin{aligned}
        \begin{bmatrix}
           \mathcal{I} \vf \\  
           \mathcal{K}\vf \\
           \mathcal{J}\vf 
           \end{bmatrix}
           =   
           \overbrace{
           \begin{bmatrix}
               w_2^2 & - \frac{w_1 w_2 + w_2 w_1}{2} & w_1^2 \\
               -w_2 w_1 & \frac{w_!^2 - w_2 ^2}{2} & w_! w_2 \\
               w_1^2 & \frac{w_1 w_2 + w_1 w_2}{2} & w_2^2 
           \end{bmatrix}}^{W}
           \begin{bmatrix}
               \mathcal{R}f_{11} \\
               \mathcal{R}f_{12} \\
               \mathcal{R}f_{22}
           \end{bmatrix}
    \end{aligned}
\end{equation*}
Note that matrix $W$
is invertible, so we can solve this system of equations to get $\mathcal{R} f_{11}$, $\mathcal{R} f_{12}$, and $\mathcal{R} f_{22}$, and further applying the inverse Radon transform, we can get the components of $\vf$.

To numerically validate the results, we implement the algorithm as in the earlier vector field case. To avoid repetition, we show the reconstruction for the phantom Ph2 and Ph3, with and without noise. Other cases can also be handled analogously. The parameters for these phantoms are given in Table \ref{tab:compact_multidisk_params_tensor_field} and \ref{tab:phantom_params_tensor_field}.
\begin{table}[H]
    \centering
    \begin{tabular}{@{}lccc@{}}
        \toprule
        \textbf{Components} & \textbf{Centers ($a, b$)} & \textbf{Radii ($r$)} & \textbf{Intensities ($val$)} \\
        \midrule
        $f_1$/$f_{11}$ & $(-0.2, 0.1)$, $(0.15, 0.15)$, $(0.0, -0.15)$ & $0.25, 0.3, 0.3$ & $3, 3.5, 4$ \\
        $f_2$/$f_{12}$ & $(0, 0.2)$, $(0.2, -0.1)$, $(-0.2, -0.15)$ & $0.2, 0.25, 0.35$ & $2, 5, 1.5$ \\
        $f_{22}$& $(0, 0.2)$, $(-0.1, -0.1)$, $(0.1, -0.05)$ & $0.15, 0.2, 0.25$ & $6, 4, 2.5$ \\
    \bottomrule
    \end{tabular}
      \caption{Parameter sets are ordered sequentially corresponding to Disks $[1, 2, 3]$.}
\label{tab:compact_multidisk_params_tensor_field}
\end{table}
\begin{table}[H]
    \centering
    \begin{tabular}{@{}lccc@{}}
        \toprule
        \textbf{Components} & \textbf{Center ($a, b$)} & \textbf{Rotation angle} & \textbf{Intensity ($val$)} \\
        \midrule
        $f_1$/$f_{11}$ & $(0.00, 0.00)$ & $10^\circ$ & $1.00$ \\
        $f_2$/$f_{12}$ & $(-0.10, -0.10)$ & $45^\circ$ & $0.95$ \\
        $f_{22}$ & $(0.00, 0.05)$ & $160^\circ$ & $1.00$ \\
        \bottomrule
    \end{tabular}    
    \caption{Phantom position and orientation parameters}
\label{tab:phantom_params_tensor_field}
\end{table}

Figure \ref{fig:Rec_tensor_without_noise_Ph2} and \ref{fig:Rec_tensor_Ph3_without_noise} shows the reconstruction of the $f_{11}, f_{12}$ and $f_{22}$ for Ph2 and Ph3 in the absence of noise. This shows that the algorithm can reconstruct sharp discontinuities for both convex and non-convex geometries with almost no artifacts. Figure \ref{fig:Rec_tensor_with_noise_Ph2} and \ref{fig:Rec_tensor_with_noise_Ph3} show that we obtain a nice reconstruction of the components of $\vf$ in the presence of $5 \%, 10 \%$ and $20 \%$ noise. There are sufficient granular artifacts but it can capture the basic features despite having $20 \%$ noise in the data. 
Table \ref{table;2} gives the relative error for various cases, including PH2 and PH3 with and without noise.
\begin{table}[h!] 
\centering 
\begin{tabular}{ |p{1.6 cm}||p{0.6cm}||p{1.6cm}|p{1.6cm}|p{1.7cm}|p{1.7cm}|  }
 \hline
Phantoms & $\vf$ & No noise & 5\% noise  & 10\% noise & 20\% noise\\[3pt]
 \hline
PH2   & $f_{11}$   & 9.88\% &	22.47\%	&  41.24\%	& 81.51\%\\[3pt]
 \hline
PH2   &  $f_{12}$  & 11.90\% &	26.71\%	& 49.22\%	& 95.02\%\\[3pt]
 \hline
 PH2  &  $f_{22}$  & 13.19\%	& 49.83\%	& 98.28\%	& 193.69\%\\[3pt]
  \hline
 PH3    & $f_{11}$ & 17.50\%	& 36.71\%	& 65.60\%	&   128.51\%\\[3pt]
 \hline 
 PH3  &  $f_{12}$ &  18.98\%	& 23.52\%	& 33.59\%	& 58.87\%\\[3pt]
  \hline
 PH3    & $f_{22}$ &  19.19\% & 27.82\%	& 44.35\%	& 81.59\%\\[3pt]
 \hline 
\end{tabular}
 \caption{Relative errors in the reconstructions of $f_{11}$, $f_{12}$, and $f_{22}$ from $\Lc\vf$, $\Tc\vf$, and $\Mc\vf$.}
\label{table;2}
\end{table}

  \begin{figure}[H] 
\centering
{\includegraphics[width= 0.7 \textwidth]{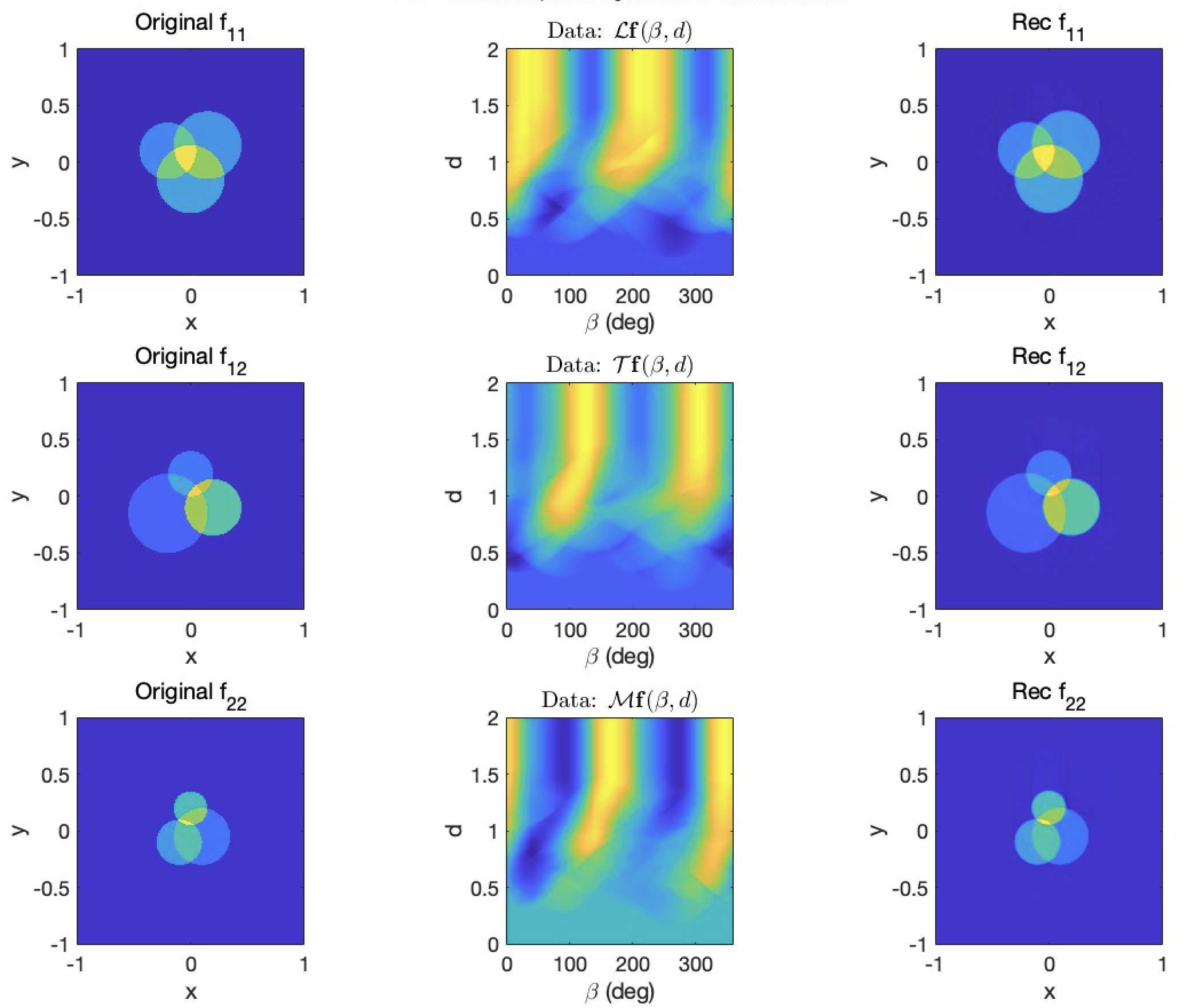}}
    \caption{Reconstructed 2-tensor field $\vf$ (Phantom 2) from $\Lc\vf$, $\Tc \vf$, and $\Mc \vf$.}
    \label{fig:Rec_tensor_without_noise_Ph2}
  \end{figure}
  \begin{figure}[H] 
\centering
{\includegraphics[width=  0.70 \textwidth]{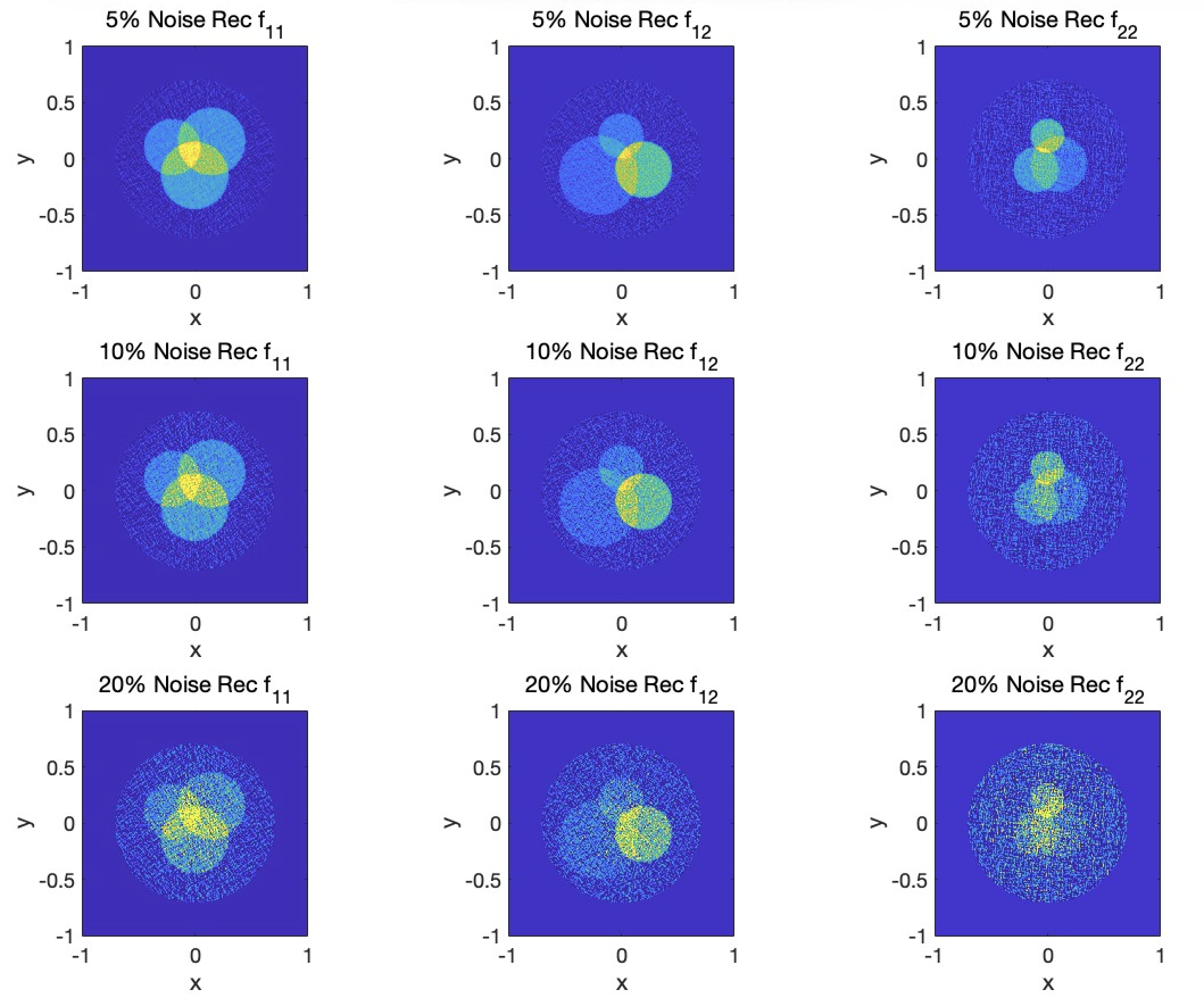}}
    \caption{Reconstructed 2-tensor field $\vf$ (Phantom 2) from $\Lc\vf$, $\Tc \vf$, and $\Mc \vf$ with noisy data.}
    \label{fig:Rec_tensor_with_noise_Ph2}
  \end{figure}
\vspace{-3 mm}
  \begin{figure}[H] 
\centering
{\includegraphics[width= .7 \textwidth]{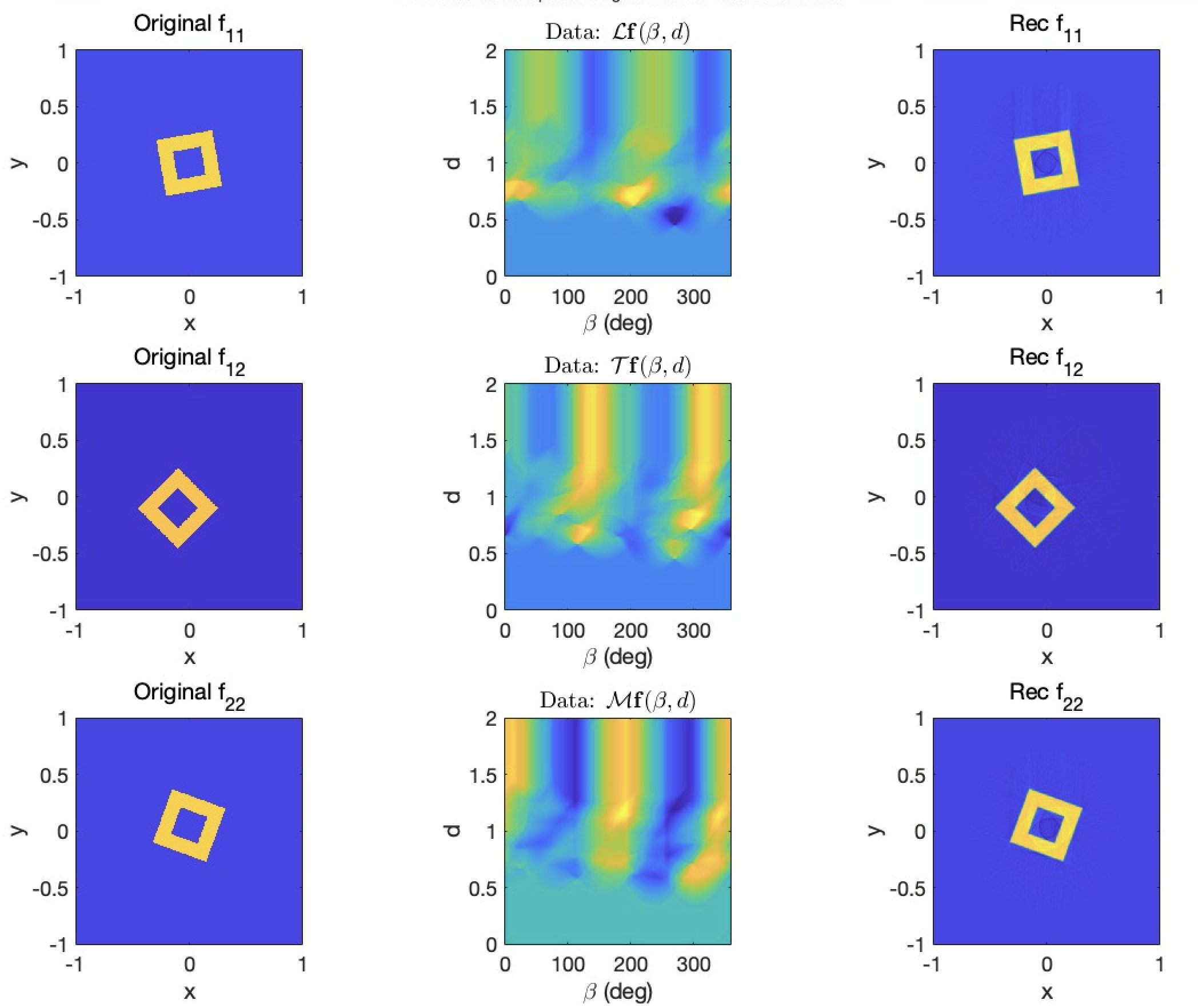}}
    \caption{Reconstructed 2-tensor field $\vf$ (Phantom 3) from $\Lc\vf$, $\Tc \vf$, and $\Mc \vf$.}\label{fig:Rec_tensor_Ph3_without_noise}
  \end{figure}
  \vspace{-3 mm}
  \begin{figure}[H] 
\centering
{\includegraphics[width= .8 \textwidth]{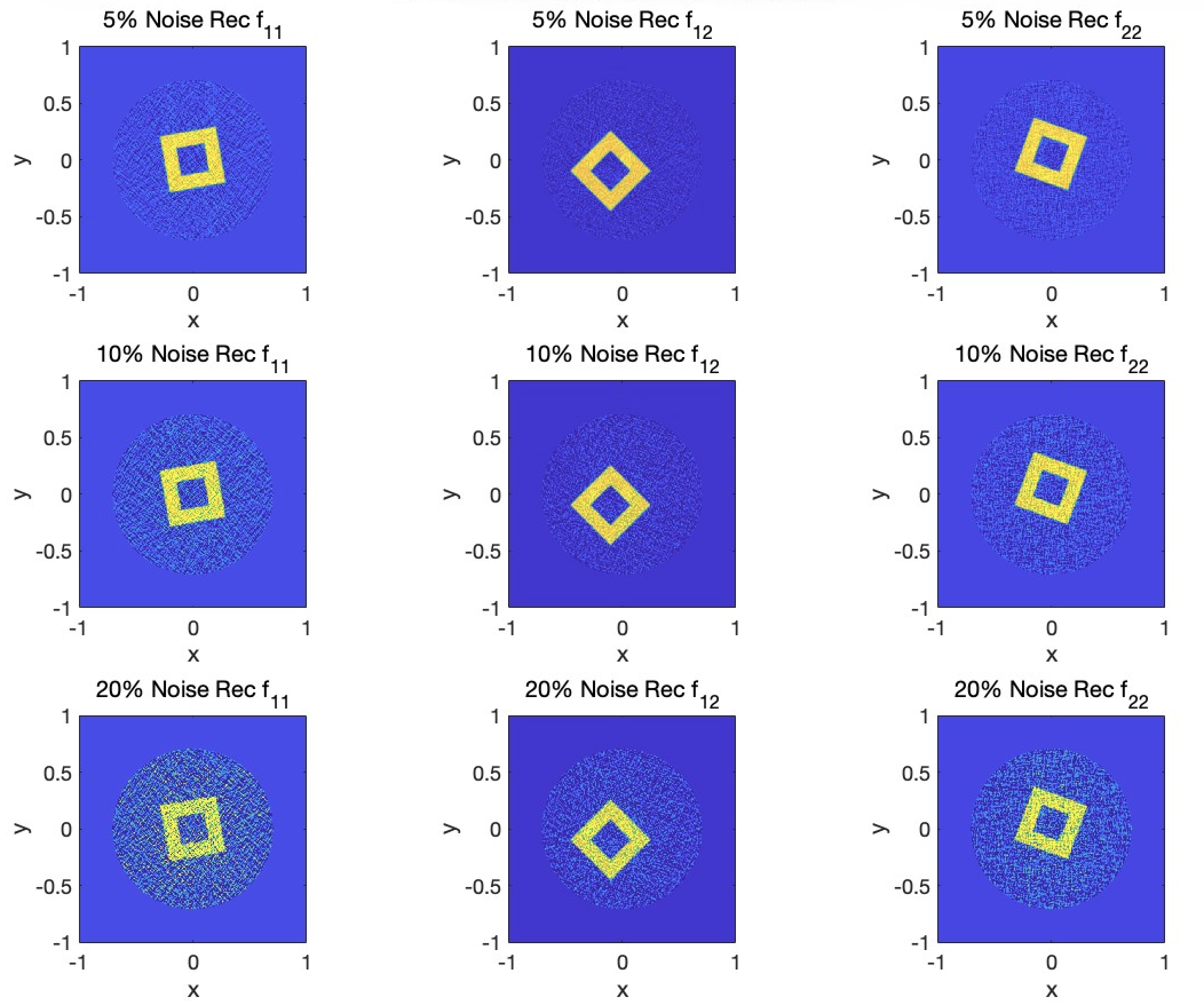}}
    \caption{Reconstructed 2-tensor field $\vf$ (Phantom 3) from $\Lc\vf$, $\Tc \vf$, and $\Mc \vf$ with noisy data.}\label{fig:Rec_tensor_with_noise_Ph3}
  \end{figure}
\vspace{- 2mm}
\section{Acknowledgements}\label{sec:acknowledgements}

Rahul Bhardwaj gratefully acknowledges the partial financial support provided by the FIST programme of the Department of Science and Technology (DST), Government of India, under Grant No.~SR/FST/MS-I/2018/22(C). 

\vspace{.25cm}

    
    
    


 \noindent \textbf{Data availability statement.} \ No datasets were generated or analyzed during the current study; therefore, data sharing is not applicable.

     \vspace{.25cm}
 \noindent \textbf{Conflict of interest.} \
 The authors declare that they have no conflicts of interest regarding the research, authorship, and/or publication of this article.
\bibliographystyle{alpha}
\bibliography{references}

\end{document}